\documentclass[12pt]{article}
\usepackage{amsmath,amsfonts,amssymb,amsxtra,latexsym,amscd,enumerate,amsthm,verbatim}

\usepackage{graphicx,xcolor}

\usepackage{enumitem}

\usepackage[margin=1.40in]{geometry}
\numberwithin{equation}{section}

\newcommand{\R}{\mathbb{R}}

\newcommand{\Z}{\mathbb{Z}}
\newcommand{\eps}{\epsilon}
\newcommand{\us}{\mathbb{S}^2}
\newcommand{\N}{\mathcal{N}}
\newcommand{\pr}{\mathbb{P}}
\newcommand{\ct}{\cos\theta}
\newcommand{\st}{\sin\theta}
\newcommand{\bV}{{\bf V}}

\newtheorem{theorem}{Theorem}[section]
\newtheorem{lemma}[theorem]{Lemma}
\newtheorem{proposition}[theorem]{Proposition}

\newtheorem{definition}[theorem]{Definition}
\newtheorem{remark}[theorem]{Remark}

\begin{document}
\date{}
\title{Refined asymptotics of the steady Navier Stokes equation around small Landau solutions}

\author{Hao Jia\thanks{School of Mathematics, University of Minnesota. Email: \texttt{jia@umn.edu}. Supported in part by NSF DMS 2245021 and NSF DMS 2453270}\and
Vladim\'ir \v{S}ver\'ak\thanks{School of Mathematics, University of Minnesota. Email: \texttt{sverak@umn.edu}. Supported in part by NSF DMS 2247027}
}

\maketitle

\setcounter{tocdepth}{1}

\abstract{In this paper we study the large distance asymptotics of small steady solutions of the 3d Navier Stokes equation in exterior domains. It was proved by Korolev and the second author \cite{SverakKorolev} that the leading term is given by the Landau solution, and it was conjectured that the next order term should be $O(1/|x|^2)$ as $x\to\infty$. We confirm that this is indeed the case and we compute the next order asymptotics in terms of eigenvalues of a suitably constructed linearized operator around the Landau solution on the unit sphere. While the decay of some of the terms is precisely $O(1/|x|^2)$, the the decay of other terms is slightly accelerated.}

\tableofcontents

\section{Introduction and main results}
Consider the steady Navier-Stokes equation in $\R^3\setminus B_1$, with $B_1=\{x\,,|x|<1\}$:
\begin{equation}\label{main1}
\begin{split}
-\Delta u+u\cdot\nabla u+\nabla p&=0,\\
{\rm div}\,u&=0,
\end{split}
\end{equation}
 where $u:\R^3\backslash B_1\to \R^3, p:\R^3\backslash B_1 \to \R$ are the velocity and pressure fields of the incompressible fluid flow. 
 Equation \eqref{main1} can be solved with suitable boundary conditions on $\partial(\R^3\setminus B_1)$ and at $\infty$. We consider the case $u(x)\to 0$ as $x\to\infty$. For the sake of concreteness, we assume that $u\in \dot{H}^1\cap L^6(\R^3\backslash B_1)$, although stronger decay conditions may be imposed in some cases. For simplicity we will also assume the natural ``no flux to infinity" condition 
 \begin{equation}\label{flux_zero}\int_{\partial B_R} u\cdot n=0
 \end{equation}
 for some (and hence all) $R>1$. This is not strictly necessary\footnote{This is essentially because a suitable multiple of the field $x/|x|^3$ can accommodate the possible non-zero flux, and in the small data situation the non-linearity does not present a serious problem for suitably decomposing the solution.}, but the main interest is in the natural situation~\eqref{flux_zero}, which already illustrates all the essential points.
 The specific boundary condition $u_0=u|_{\partial(\R^3\setminus B_1)}$ will not play an important role in our results (but our assumptions imply that it should be small).

The study of the steady Navier-Stokes equations has a long history, starting with the pioneering work of Leray \cite{leray1933etude} who established, for given boundary data, the existence of steady state solutions with finite Dirichlet energy. Such solutions are also smooth following standard elliptic regularity theory, thanks to the subcritical nature of the steady Navier-Stokes equations in three dimensions. We remark that uniqueness for fixed large boundary conditions is generally not expected, see e.g. \cite{Yud67} and section 2.5 in \cite{tsai2018lectures}. 

In general, determining the large distance asymptotics is
difficult.  
In contrast to the regularity problem, the large distance asymptotics problem is more difficult in lower dimensions (somewhat similarly to scattering for dispersive equations), with the four-dimensional case being ``critical'' (see \cite{jia2018asymptotics}) and considerably easier than the physically significant two- and three-dimensional cases.
 We refer to \cite{amick1988leray,babenko1973stationary,cannone2004smooth,deuring2000asymptotic,farwig2009asymptotic,finn1965exterior,galdi1994introduction,guillod2015steady,korobkov2015solution,landau1944new,leray1933etude,nazarov2000steady,tian1998one}, the books \cite{constantin1988navier,galdi1994introduction,tsai2018lectures} and the recent survey \cite{korobkov2023stationary_ren}
 as well as references therein for some classical results and more recent significant developments in this large area. 
 
 
 One open problem for \eqref{main1} is whether all smooth finite Dirichlet energy solutions $u$ vanishing at $\infty$ satisfy the natural asymptotics
 \begin{equation}\label{Int2}
     |u(x)|=O(1/|x|),\quad {\rm as}\,\,x\to\infty. 
 \end{equation}
 One can also ask about the precise asymptotics as $x\to\infty$ for solutions of~\eqref{main1} that vanish at $\infty$. In general, this seems difficult for large solutions. 

 For small solutions, using perturbative methods, it is not hard to prove \eqref{Int2} under quite general conditions. However, even in this case, there are some interesting open problems concerning the more precise asymptotics of the solution as $x\to\infty$ that we will address in this paper. 

There is a remarkable family of explicit solutions $U^b, P^b$ with $b\in\R^3$ to \eqref{main1}, the Landau solutions, which satisfy in $\R^3$,
\begin{equation}\label{Int3}
    \begin{split}
-\Delta U^b+U^b\cdot\nabla U^b+\nabla P^b&=b\,\delta(x),\\
{\rm div}\,U^b&=0,
\end{split}
\end{equation}
in the sense of distributions. 

The Landau solutions are given as follows. Assume without loss of generality $b=|b| \begin{pmatrix}
    0\\0\\1
\end{pmatrix}$. We use the standard spherical coordinate $(r,\theta, \phi)$ with 
\begin{equation}\label{Int3.1}
    x_1=r\sin\theta \,\cos\phi,\,\,x_2=r\sin\theta\,\sin\phi,\,\,x_3=r\cos\theta,
\end{equation}
and the vector fields
\begin{equation}\label{Int3.10}
\begin{split}
  & e_r:=(\sin\theta \,\cos\phi,\sin\theta\,\sin\phi,\cos\theta),\,\, e_\theta:=(\cos\theta\,\cos\phi,\cos\theta\,\sin\phi,-\sin\theta),\\
 &  e_\phi:=(-\sin\phi,\cos\phi, 0).
\end{split}   
\end{equation}
Write for $x\in\R^3\backslash\{0\}$,
\begin{equation}\label{Int3.2}
    U^b(x):=\frac{1}{|x|}V^b(x/|x|)e_\theta+\frac{1}{|x|}F^b(x/|x|)e_r,
\end{equation}
then we have the formulae
\begin{equation}\label{Int3.3}
\begin{split}
  &  V^b=-\frac{2\epsilon\sin\theta}{1-\epsilon\cos\theta},\quad F^b=2\Big[\frac{1-\epsilon^2}{(1-\epsilon\cos\theta)^2}-1\Big],\\
  &|b|=f(\epsilon):=16\pi\Big[\frac{1}{\epsilon}+\frac{1}{2\epsilon^2}
  \log\frac{1-\epsilon}{1+\epsilon}+\frac{4\epsilon}{3(1-\epsilon^2)}\Big],\quad \epsilon\in(0,1).  
    \end{split}
\end{equation}
Notice that $f$ is increasing and maps $(0,1)$ to $(0,\infty)$ and $f(\epsilon)\approx \epsilon$ for $0<\epsilon\ll1$. The formula \eqref{Int3.3} can be found in \cite{SverakKorolev} and in page 207 of \cite{Batchelor}

Landau solutions are scale invariant under the natural scaling of \eqref{main1} for any $\lambda>0$,
\begin{equation}
    u\to u_\lambda(x):=\lambda u(\lambda x),\quad p\to p_\lambda(x):=\lambda^2p(\lambda x). 
\end{equation}
In particular, Landau solutions $U^b$ are $-1$-homogeneous and satisfy \eqref{Int2}. In \cite{sverak2011landau} it was proved that Landau solutions are the only smooth solutions to the steady Navier Stokes in $\R^3\backslash\{0\}$ which are scale invariant. 

For a smooth solution $u,p$ to \eqref{main1}, define for $i,j\in\{1,2,3\}$ and $|x|>1$,
\begin{equation}\label{Int4}
    T_{ij}(x):=-(\partial_{x_j}u_i+\partial_{x_i}u_j)+u_iu_j+p\,\delta_{ij},
\end{equation}
then simple integration by part argument shows that for any $R_2>R_1>1$,
\begin{equation}\label{Int5}
    b_i:=\int_{|x|=R_1}T_{ij}n_j\,d\sigma=\int_{|x|=R_2}T_{ij}n_j\,d\sigma.
\end{equation}
For solutions to \eqref{main1} satisfying for $|x|>1$,
\begin{equation}\label{Int6}
    |u(x)|\leq c_\ast/|x|,
\end{equation}
with a sufficiently small $c_\ast>0$, it was proved by Korolev and the second author \cite{SverakKorolev} that the leading order asymptotic is given by $U^b$, in the sense that for some $\gamma\in(1,2)$, 
\begin{equation}\label{Int7}
    \big|u(x)-U^b(x)\big|=O(1/|x|^\gamma), \quad x\to\infty,
\end{equation}
where $b\in\R^3$ is given by \eqref{Int4}-\eqref{Int5}. 

It is worthwhile to point out that at the leading order term $\sim O(\frac{1}{|x|})$, the Landau solutions behave differently from solutions to the corresponding linear Stokes equation. Thus the result \cite{SverakKorolev} shows that the nonlinearity has a significant effect on the large distance asymptotic of solutions to \eqref{main1} even for small solutions. We expect that the Landau solutions should play an important role in understanding asymptotic behavior of general solutions to \eqref{main1}, not just the small ones. 

We note that in \eqref{Int7}, $\gamma$ can be chosen arbitrarily close to $2$ but the endpoint $\gamma=2$ is excluded in \cite{SverakKorolev}. In addition, $c_\ast$ depends on $\gamma$ and may shrink to $0$ as $\gamma$ approaches $2$. In \cite{SverakKorolev} it was conjectured that the asymptotic expansion \eqref{Int7} should still hold for $\gamma=2$. However, simple examples in~\cite{SverakKorolev} show that the perturbative methods used there are not sufficient to settle the problem.

In this paper, we confirm this expectation and prove the following theorem. 
\begin{theorem}\label{MTh1}
There exists $\delta_0\in(0,1)$ such that the following statement holds. Suppose that $u,p$ is a smooth solution to \eqref{main1} satisfying for $|x|>1$,
\begin{equation}\label{MTh1.1}
    |u(x)|\leq \delta_0/|x|,
\end{equation}
together with~\eqref{flux_zero}.
Let $b\in\R^3$ be given as in \eqref{Int4}-\eqref{Int5} and $\epsilon$ be defined as in \eqref{Int3.3}. Then we have the following conclusions:

\medskip
\noindent
\quad (I) There exists an eight dimensional space $E_b\subseteq C^\infty(\us,\R^3)$ and subspaces $E_b^1, \,E_b^2$, $E_b^3\subseteq E_b$ depending only on $b$, such that 
\begin{equation}\label{MTh1.2}
    E_b=E_b^1\oplus E_b^2\oplus E_b^3, \quad {\rm dim }\,E^1_b=4, \,\, \quad {\rm dim }\,E^2_b=2, \quad {\rm dim }\,E^3_b=2. 
\end{equation}

\medskip
\noindent
\quad(II) We have the following expansion for $x\in\R^3\backslash B_1$,
\begin{equation}\label{MTh1.3}
\begin{split}
    u(x)=&U^b(x)+\frac{\Phi_1(x/|x|)}{|x|^2}\\
    &+\frac{1}{|x|^{\alpha_2+1}}\Big[\cos(\beta_2\log|x|)\Phi_{2,1}(x/|x|)+\sin(\beta_2\log|x|)\Phi_{2,2}(x/|x|)\Big]\\
    &+\frac{1}{|x|^{\alpha_3+1}}\Big[\cos(\beta_3\log|x|)\Phi_{3,1}(x/|x|)+\sin(\beta_3\log|x|)\Phi_{3,2}(x/|x|)\Big]\\
    &+O(|x|^{-5/2}),
    \end{split}
    \end{equation}
    where the constants $\alpha_2, \beta_2, \alpha_3, \beta_3$ depend only on $b$ and satisfy
    \begin{equation}\label{MTh1.4}
    \begin{split}
        &\alpha_2=1+\frac{1}{15}\epsilon^2+O(\epsilon^3), \quad\beta_2=O(\epsilon^3),\\
        &\alpha_3=1+\frac{4}{15}\epsilon^2+O(\epsilon^3),\quad\beta_3=O(\epsilon^3).
        \end{split}
    \end{equation}
    Moreover, $\Phi_1\in E_b^1$, $\big\{\Phi_{2,1}, \Phi_{2,2}\big\}\subset E_b^2$ and $\big\{\Phi_{3,1}, \Phi_{3,2}\big\}\subset E_b^3$. 
\end{theorem}

\bigskip

 \begin{remark}  We have the following remarks.

    \begin{enumerate}

    \item The expansion \eqref{MTh1.3} and the formulae \eqref{MTh1.4} imply that the presence of the Landau solution accelerates the decay of the next order terms in comparison with the linear Stokes equation, with all terms decaying at $1/|x|^2$ or faster rates. It would be interesting to come up with a more physical explanation for this phenomenon (our proof is based on precise computations). 
 
        \item Careful numerical calculations suggest that in \eqref{MTh1.4}, $\beta_2=\beta_3=0$, and there is no ``oscillation" in the asymptotic expansion \eqref{MTh1.3}. A rigorous proof is still open.

\item The bound $O(|x|^{-5/2})$ on the higher order term in \eqref{MTh1.3} is taken for conveniences, and a more careful argument gives $O(|x|^{-\mu})$ for some positive $\mu<3$ depending on $|b|$, and $\mu\to3-$ as $|b|\to0+$ (though one might not be able to take $\mu=3$ in the case $b=0$ due to possible logarithmic terms that could appear at $1/|x|^3$ level). 

\item The spaces $E_b^1, E_b^2$ and $E_b^3$ arise naturally from eigenspaces associated with the linearized operator around $U^b$, when appropriately reformulated on the sphere, see \eqref{S1e9}--\eqref{S1e11.1}. $E^b$ depends smoothly on the parameter $b\in\R^3$ for sufficiently small $|b|$, in the sense that the associated spectral projection operator $P^b$ to the space $E^b$ (which can be defined through a contour integral of the resolvent) depends smoothly on $b$. Moreover, as $b\to 0$, $P^b$ converges to the projection operator to the space $E_1$ given explicitly in Proposition \ref{S2P1}, which corresponds to the expansion of solutions to the linear Stokes system at the order $1/|x|^2$. In particular, Theorem \ref{MTh1} holds uniformly in the limit $b\to0$. 
  \end{enumerate}
\end{remark}

We emphasize that it is crucial to identify precisely the function spaces $E_b^1, E_b^2$ and $E_b^3$ which capture the approximately quadratically decaying terms in the expansion \eqref{MTh1.3}, since these terms can not be treated by simple perturbation arguments alone. 

Roughly speaking, the reason is as follows. The interaction of the main term $U^b$ and the quadratically decaying terms through nonlinearity produces terms which decay with the rate $O(1/|x|^4)$, which also contains a derivative.  Upon integration by parts when inverting the linear Stokes system, we obtain an integral of the type
\begin{equation}\label{Mth2}
    D(x):=\int_{|y|>1}\frac{1}{|x-y|^2}\frac{1}{|y|^3}dy,
\end{equation}
where $x$ can be assumed to be large. It is clear by simple calculations that $D(x)\sim |x|^{-2}\log |x|$ for large x. This logarithmic factor is the main reason why rough arguments treating the quadratic terms as pure perturbations only yield the less precise expansion \eqref{Int7}. Heuristically, the decay rate of Landau solutions $U^b$ is exactly ``critical" for the study of long distance asymptotics of the steady Navier Stokes equations, and could potentially change the rate of decay of small perturbations as $x\to\infty$. 

To obtain the sharper expansion \eqref{MTh1.3}, the key observation is to identify various terms in the expansion as eigenfunctions corresponding to the linearized Navier Stokes equation around $U^b$. The full analysis is quite complicated since we need to deal with vector and pressure fields which lead to many dimensions. The main idea though can already seen from the following simple example. 

Consider the equation 
\begin{equation}\label{MTh3}
    -\Delta\varphi+\frac{a(x/|x|)}{|x|^2}\varphi(x)=0, \quad{\rm for}\,\,x\in\R^3\backslash B_1.
\end{equation}
We assume that $a\in C^\infty(\us)$ and the smooth solution $u$ satisfy the bound for some $C>0$ and all $|x|\ge1$,
\begin{equation}\label{MTh4}
    |u(x)|\leq C/|x|.
\end{equation}
Using spherical coordinate $r>0, \,\,\omega\in\us$, we can reformulate equation \eqref{MTh3} as
\begin{equation}\label{MTh5}
    -\partial_r^2\varphi(r,\omega)-\frac{2}{r}\partial_r\varphi(r,\omega)-\frac{1}{r^2}\Delta_{\us}\varphi(r,\omega)+\frac{1}{r^2}a(\omega)\varphi(r,\omega)=0.
\end{equation}
In the above, $\Delta_{\us}$ denotes the spherical Laplacian on $\us$. Define the operator $A: H^2(\us)\to L^2(\us)$ as 
\begin{equation}\label{MTh6}
    h\in H^2(\us)\to -\Delta_{\us}h+a(\omega)h. 
\end{equation}
$A$ is a self adjoint operator with real discrete eigenvavlues $\lambda_j\in\R, j\in\Z\cap[1,\infty)$ and $\lambda_j\to\infty$ as $j\to\infty$. For $\lambda_j>0$, assume that $\lambda_j=\alpha(\alpha-1)$ for some $\alpha>1$ and that $\Phi_j$ is an eigenfunction for $\lambda_j$. Simple calcultions show that 
\begin{equation}\label{MTh7}
    \varphi(r,\omega)=r^{-\alpha}\Phi_j(\omega)
\end{equation}
is a solution to \eqref{MTh5}. A slightly more detailed analysis shows that, in fact, the spectral property of $A$ completely determines the asymptotic expansion of solutions to \eqref{MTh5} as $x\to\infty$. 

Motivated by the example \eqref{MTh3}, we reformulate \eqref{main1} and the linearized Navier Stokes equation around $U^b$ in spherical coordinates, and define the corresponding operators on $\us$. The full equation \eqref{main1} can be viewed as a dynamical system in $r>0$ (more precisely we will work with the variable $s=-\log r$) with values in $L^2(\us)$. The resulting operator defined on the sphere is non self-adjoint and acts on fields with $6$ components. Using general spectral theory, we reduce the analysis to the calculation of the perturbed eigenvalues (due to the Landau solution) from the the eigenvalues $\lambda\in\{0,1\}$ of the unperturbed operator (corresponding to the linear Stokes system reformulated on the unit sphere). We can show using symmetries of the equation \eqref{main1} that $\lambda=0$ is invariant under small perturbations by the Landau solutions. 

The computation of the perturbed eigenvalues around $\lambda=1$ is the heart of our analysis. Since the eigenspace corresponding to $\lambda=1$ of the unperturbed operator is $8$ dimensional, one can expect that the eigenvalue $\lambda=1$ may split into several curves of eigenvalues under small perturbations. 

We first use the structure of the perturbation to show that $\lambda=1$ remains an eigenvalue for the perturbed operator with a four-dimensional eigenspace. Translation symmetries produce three linearly independent eigenvectors, but the fourth eigenvector is quite nontrivial and is given by a ``pure swirl" field, see section \ref{sec:pureswirl}.

To study the behavior of the remaining four dimensional subspace, we note that $U^b$ has an additional rotational symmetry, which allows to take a Fourier transform in $\phi$ and separately consider the modes $m=1$ and $m=2$ cases. To facilitate the calculations, we formulate a general method of tracking the eigenvalues, by describing precisely the invariant subspace as a graph over the unperturbed eigenspace up to any order. The task is then reduced to a finite dimensional calculation that is somewhat similar to what is usually done in quantum mechanics. We hope that the general method we introduced here may be of use for other perturbative problems where precise computation of eigenvalues with high multiplicities is required. 

\subsection{Notations}
Below we shall use 
\begin{equation}\label{not0}
   \bV^b:=V^be_\theta. 
\end{equation}
For simplicity of notations, we sometimes use the shorthands when there is no possibility of confusion
\begin{equation}\label{not1}
   V:=V^\epsilon:=V^b,\quad \bV:=\bV^\epsilon:=\bV^b,\quad F:=F^\epsilon:=F^b.
\end{equation}

\section{Main equations}\label{sec:maineq}
Suppose $u: \R^3\backslash B_1\to \R^3, p: \R^3\backslash B_1\to \R$ are smooth functions satisfying equation \eqref{main1} and the decay condition as $x\to\infty$,
\begin{equation}\label{main2}
\big|u(x)\big|=O\big(1/|x|\big),\qquad \big|p(x)\big|=O\big(1/|x|^2\big). 
\end{equation}
Writing for $x=ry, \,r>0, \,y\in \mathbb{S}^2$, 
\begin{equation}\label{S1e1}
u(x)=\frac{1}{r}\big(v(r,y)+f(r,y)e(y)\big)\qquad p(x)=\frac{1}{r^2}q(r,y), 
\end{equation}
where $v$ is tangent to the sphere and $e(y)=y$ for $y\in\mathbb{S}^2$. Then the Navier Stokes equations can be re-written as
\begin{equation}\label{S1e2}
\begin{split}
&-\partial_r^2v+\frac{1}{r^2}(-\Delta v+v\cdot\nabla v +rf\partial_rv+\nabla(q-2f))=0,\\
&-\partial_r^2f+\frac{1}{r^2}(-\Delta f+v\cdot\nabla f-|v|^2-2r\partial_rf-f^2+rf\partial_rf-2q+r\partial_rq)=0,\\
&{\rm div}\,v+f+r\partial_rf=0.
\end{split}
\end{equation}
The detailed derivation of \eqref{S1e2} can be found in \cite{sverak2011landau}.  

The system \eqref{S1e2} has a three parameter family of solutions which are smooth on $\us$ and are independent of $r>0$, called Landau solutions. Denote these solutions as $V^b(y), F^b(y), y\in \us$ with $b\in
\R^3$. We are interested in the behavior of solutions which are close to the Landau solutions. The linearized system is 
\begin{equation}\label{S1e3}
\begin{split}
&-\partial_r^2v+\frac{1}{r^2}(-\Delta v+{\bf V}^b\cdot\nabla v+v\cdot\nabla {\bf V}^b+rF^b\partial_rv+\nabla(q-2f))=0,\\
&-\partial_r^2f+\frac{1}{r^2}\big(-\Delta f+{\bf V}^b\cdot\nabla f+v\cdot\nabla F^b-2v\cdot {\bf V}^b-2r\partial_rf-2fF^b\\
&\qquad\qquad\qquad+\,r\partial_rfF^b-2q+r\partial_rq\big)=0,\\
&{\rm div}\,v+f+r\partial_rf=0.
\end{split}
\end{equation}
We assume that the perturbation $(v,f)$ satisfy $|v,f|=o(1/r)$ as $r\to\infty$. We now reformulate the system \eqref{S1e3} as a system of evolution equations in $\log r$ as follows. 
Set
\begin{equation}\label{S1e4}
\begin{split}
&\xi(s,y):=v(r,y), \quad \xi'(s,y):=r\partial_rv(r,y),\\
&\vartheta(s,y):=f(r,y),\quad \vartheta'(s,y):=r\partial_rf(r,y),\,\, s=-\log r, \,\,r\in(R,\infty), \,\,s_0:=-\log R.
\end{split}
\end{equation}
In the variables $\xi, \xi', \vartheta,\vartheta'$, the system \eqref{S1e3} becomes
\begin{equation}\label{S1e5}
\begin{split}
&\partial_s\xi(s,y)=-\xi'(s,y),\\
&\partial_s\xi'(s,y)=-\xi'+\big(\Delta\xi-{\bf V}^b\cdot\nabla\xi-\xi\cdot\nabla {\bf V}^b-F^b\xi'-\nabla(q-2\vartheta)\big)(s,y),\\
&\partial_s\vartheta(s,y)=-\vartheta'(s,y),\\
&\partial_s\vartheta'(s,y)=\vartheta'(s,y)+\big(\Delta\vartheta-{\bf V}^b\cdot\nabla\vartheta-\xi\cdot\nabla F^b+2\xi \cdot {\bf V}^b\\
&\qquad\qquad \qquad\qquad\qquad+2\vartheta F^b-\vartheta'F^b+2q+\partial_sq\big)(s,y),\\
&{\rm div}\,\xi(s,y)+\vartheta(s,y)+\vartheta'(s,y)=0,\qquad {\rm for}\,\,y\in\us, s\in(-\infty,s_0].
\end{split}
\end{equation}
It is convenient to also introduce the vorticity-like functions $\omega(s,y)$ and $\omega'(s,y)$ through 
\begin{equation}\label{S1e6}
d\xi(s,y):=\omega(s,y)\Omega_0,\qquad d\xi'(s,y):=\omega'(s,y)\Omega_0,
\end{equation}
where $\Omega_0$ is the standard volume form on $\us$. 

To eliminate the term $\partial_sq$, we define
\begin{equation}\label{S1e7}
\vartheta^{\ast}:=\vartheta'-q.
\end{equation}
Then the system \eqref{S1e5} can be reformulated as 
\begin{equation}\label{S1e8}
\begin{split}
&\partial_s\xi(s,y)=-\xi'(s,y),\\
&\partial_s\xi'(s,y)=-\xi'(s,y)+\big(\Delta\xi-{\bf V}^b\cdot\nabla\xi-\xi\cdot\nabla {\bf V}^b-F^b\xi'-\nabla(q-2\vartheta)\big)(s,y),\\
&\partial_s\vartheta(s,y)=-\vartheta^\ast(s,y)-q(s,y),\\
&\partial_s\vartheta^\ast(s,y)=\big(\vartheta^\ast+3q+\Delta\vartheta-{\bf V}^b\cdot\nabla\vartheta-\xi\cdot\nabla F^b+2\xi \cdot {\bf V}^b\\
&\qquad\qquad\qquad+2\vartheta F^b-\vartheta^\ast F^b-q F^b\big)(s,y),\\
&{\rm div}\,\xi+\vartheta+\vartheta^\ast+q=0,\qquad {\rm for}\,\,y\in\us, s\in(-\infty,s_0].
\end{split}
\end{equation}
The asymptotics of $u$ is now transformed to the study of the solutions to the system \eqref{S1e8} as $s\to-\infty$, which depends on the spectral property of the associated linear operator. Define
\begin{equation}\label{S1e9}
L\begin{pmatrix}
\xi\\
\xi'\\
\vartheta\\
\vartheta^\ast
\end{pmatrix}=\begin{pmatrix}
-\xi'\\
-\xi'+\Delta\xi-{\bf V}^b\cdot\nabla\xi-\xi\cdot\nabla {\bf V}^b-F^b\xi'-\nabla(q-2\vartheta)\\
-\vartheta^\ast-q\\
\vartheta^\ast+3q+\Delta\vartheta-{\bf V}^b\cdot\nabla\vartheta-\xi\cdot\nabla F^b+2\xi \cdot {\bf V}^b+2\vartheta F^b-\vartheta^\ast F^b-q F^b
\end{pmatrix}
\end{equation}
and the closely related
\begin{equation}\label{S1e10}
L_0\begin{pmatrix}
\xi\\
\xi'\\
\vartheta\\
\vartheta^\ast
\end{pmatrix}=\begin{pmatrix}
-\xi'\\
-\xi'+\Delta\xi-\nabla(q-2\vartheta)\\
-\vartheta^\ast-q\\
\vartheta^\ast+3q+\Delta\vartheta
\end{pmatrix},
\end{equation}
where the function $q$ is given by ${\rm div}\,\xi+\vartheta+\vartheta^\ast+q=0,\,\, {\rm for}\,\,y\in\us.$
The operators $L, L_0$ are defined on the space
\begin{equation}\label{S1e11}
\mathcal{D}:=\big\{(\xi,\xi',\vartheta,\vartheta'):\xi\in H^3(\us), \xi'\in H^2(\us), \vartheta\in H^3(\us), \int_{\us}\vartheta\,=0, \vartheta^\ast\in H^2(\us)\big\} 
\end{equation}
and map to the space
\begin{equation}\label{S1e11.1}
X:=\big\{(f,f',g,g^\ast):f\in H^2(\us), f'\in H^1(\us), g\in H^2(\us), \int_{\us}g=0, g^\ast\in H^1(\us)\big\}.
\end{equation}
The integral condition that $\int_{\us}g=\int_{\us}\vartheta\,=0$ is motivated by the ``no outflow" condition for divergence free fields.

We can reformulate the steady Navier Stokes equation \eqref{S1e2} outside a fixed ball $B_{r_0}, r_0>0, r_0=e^{-s_0}$ as
\begin{equation}\label{MainEv}
\partial_s\begin{pmatrix}
\xi\\
\xi'\\
\vartheta\\
\vartheta^\ast
\end{pmatrix}=L\begin{pmatrix}
\xi\\
\xi'\\
\vartheta\\
\vartheta^\ast
\end{pmatrix}+\begin{pmatrix}
0\\
-\xi\cdot\nabla\xi-\vartheta\xi'\\
0\\
-\xi\cdot\nabla\vartheta+\vartheta^2+\vartheta(\vartheta^\ast+q)\end{pmatrix}, \qquad{\rm for}\,\,s\in(-\infty,s_0],
\end{equation}
where we assume ${\rm div}\,\xi+\vartheta+\vartheta^\ast+q=0$, which is the usual divergence free condition.

We summarize our calculations in the following proposition.
\begin{proposition}\label{Pe}
Fix $R>0$. Suppose that $u: \R^3\backslash B_R\to \R^3, p: \R^3\backslash B_R\to \R$ are smooth functions satisfying
\begin{equation}\label{Pe1}
\begin{split}
-\Delta u+{\bf V}^b\cdot\nabla u+u\cdot\nabla {\bf V}^b+u\cdot\nabla u+\nabla p&=0,\\
{\rm div}\,u&=0,\qquad {\rm for}\,\,x\in\R^3\backslash B_R,
\end{split}
\end{equation}
and the decay condition as $x\to\infty$ (with some $\delta\in(0,1)$)
\begin{equation}\label{Pe2}
\big|u(x)\big|=O\big(1/|x|^{1+\delta}\big),\qquad \big|p(x)\big|=O\big(1/|x|^{2+\delta}\big). 
\end{equation}

 Writing for $x=ry, \,r>0, \,y\in \mathbb{S}^2$,
\begin{equation}\label{Pe3}
u(x)=\frac{1}{r}(v(r,y)+f(r,y)e(y)),\quad p(x)=\frac{1}{r^2}q(r,y),  
\end{equation}
where $v$ is tangent to the sphere and $e(y)=y$ for $y\in\mathbb{S}^2$. Set
\begin{equation}\label{Pe4}
\begin{split}
&\xi(s,y):=v(r,y), \quad \xi'(s,y):=r\partial_rv(r,y),\quad \vartheta(s,y):=f(r,y),\\
&\vartheta'(s,y):=r\partial_rf(r,y),\quad \vartheta^{\ast}:=\vartheta'-q, \qquad s=-\log r, \,\,r\in(R,\infty), \,\,s_0:=-\log R.
\end{split}
\end{equation}
We can reformulate the steady Navier Stokes equation \eqref{Pe1} as
\begin{equation}\label{Pe6}
\partial_s\begin{pmatrix}
\xi\\
\xi'\\
\vartheta\\
\vartheta^\ast
\end{pmatrix}=L\begin{pmatrix}
\xi\\
\xi'\\
\vartheta\\
\vartheta^\ast
\end{pmatrix}+\begin{pmatrix}
0\\
-\xi\cdot\nabla\xi-\vartheta\xi'\\
0\\
-\xi\cdot\nabla\vartheta+\vartheta^2+\vartheta(\vartheta^\ast+q)\end{pmatrix}, \qquad{\rm for}\,\,s\in(-\infty,s_0],
\end{equation}
where we assume ${\rm div}\,\xi+\vartheta+\vartheta^\ast+q=0$, which is the divergence free condition. We also use \eqref{S1e9} for the definition of $L$.

\end{proposition}

%
\section{Spectral properties of $L_0$ and bounds on its associated semigroups}
Let us take $\lambda \in \mathbb{C}$ and consider the problem of inverting $L_0-\lambda$. Therefore, we need to study the equation
\begin{equation}\label{S2e1}
\begin{split}
&-\xi'-\lambda\xi=f,\\
&-\xi'+\Delta\xi-\nabla(q-2\vartheta)-\lambda\xi'=f',\\
&-\vartheta^\ast-q-\lambda\vartheta=g,\\
&\vartheta^\ast+3q+\Delta\vartheta-\lambda\vartheta^\ast=g^\ast,\\
&{\rm div}\,\xi+\vartheta+\vartheta^\ast+q=0,
\end{split}
\end{equation}
for $y\in \us$, where we assume that $f', g^\ast\in L^2(\us)$ and $f, g\in H^1(\us)$ with $\int_{\us}g=0$. We can normalize and assume that 
\begin{equation}\label{S2e1.1}
\|(f,f',g,g^\ast)\|_{X}\leq1.
\end{equation}
 After simplification, we obtain on $\us$,
\begin{eqnarray}
&&\xi'=-f-\lambda \xi,\label{S2e2.1}\\
&&-\vartheta^\ast=g+q+\lambda\vartheta,\label{S2e2.2}\\
&&\Delta\xi+\lambda(1+\lambda)\xi-\nabla(q-2\vartheta)=f'-(1+\lambda)f,\label{S2e2.3}\\
&&\Delta\vartheta+\lambda(\lambda-1)\vartheta+(\lambda+2)q=g^{\ast}-(\lambda-1)g,\label{S2e2.4}\\
&&{\rm div}\,\xi-(\lambda-1)\vartheta-g=0.\label{S2e2.5}
\end{eqnarray}
We only need to focus on the system \eqref{S2e2.3}-\eqref{S2e2.5}. Applying $d$ operator to \eqref{S2e2.3} and recalling the definition \eqref{S1e6} for $\omega$, we get 
\begin{equation}\label{S2e3}
\Delta \omega+\lambda(1+\lambda) \omega=*df'-(1+\lambda)*df.
\end{equation}
To ensure this equation is solvable, in view of the spectrum of the spherical Laplacian, we assume that $\lambda\not\in \mathbb{Z}$. Then equation \eqref{S2e3} is solvable and we have the bounds
\begin{equation}\label{S2e4}
\|\omega\|_{H^2(\us)}\lesssim_{\lambda}1.
\end{equation}
We remark that the implied constants depend delicately on $\lambda$ and we will later on derive more effective bounds. Now applying $d^\ast$ operator to \eqref{S2e2.3} we obtain
\begin{equation}\label{S2e5}
\Delta d^\ast\xi+\lambda(1+\lambda)d^\ast\xi+\Delta(q-2\vartheta)=d^\ast f'-(1+\lambda)d^\ast f.
\end{equation}
It then follows from \eqref{S2e5}, \eqref{S2e2.4}-\eqref{S2e2.5} and $d^\ast \xi=-{\rm div}\,\xi$ that
\begin{equation}\label{S2e6}
\Delta q+(\lambda+1)(\lambda+2)q=d^\ast f'-(1+\lambda)d^\ast f+\Delta g+(\lambda+1)g^\ast+(\lambda+1)g.
\end{equation}
Since $\lambda\not\in\mathbb{Z}$, \eqref{S2e6} can be solved and we have the bounds
\begin{equation}\label{S2e7}
\| q\|_{H^2(\us)}\lesssim_{\lambda}1.
\end{equation}
With the bounds \eqref{S2e7} on $q$, we can now use \eqref{S2e2.4} to solve for $\vartheta$ and then solve \eqref{S2e2.2} for $\xi$. Summarizing the above, if $\lambda\not\in \mathbb{Z}$, then the system \eqref{S2e1} is uniquely solvable for $f, f', g, g^\ast$ satisfying \eqref{S2e1.1}, and we have the bounds
\begin{equation}\label{S2e8}
\|(\xi,\xi',\vartheta,\vartheta^\ast)\|_{\mathcal{D}}\lesssim_{\lambda}1.
\end{equation}

Therefore we can conclude that the spectrum $\sigma(L_0)$ of $L_0$ is contained in $\mathbb{Z}$. For $\lambda=k\in\mathbb{Z}$, the eigenvalue equation is
\begin{eqnarray}
&&\xi'=-\lambda \xi,\label{S2e9.1}\\
&&-\vartheta^\ast=q+\lambda\vartheta,\label{S2e9.2}\\
&&\Delta\xi+\lambda(1+\lambda)\xi-\nabla(q-2\vartheta)=0,\label{S2e9.3}\\
&&\Delta\vartheta+\lambda(\lambda-1)\vartheta+(\lambda+2)q=0,\label{S2e9.4}\\
&&{\rm div}\,\xi-(\lambda-1)\vartheta=0.\label{S2e9.5}
\end{eqnarray}
Using argument as above, we can derive from \eqref{S2e9.2}-\eqref{S2e9.5} the equation
\begin{equation}\label{S2e10}
\Delta q+(\lambda+1)(\lambda+2)q=0.
\end{equation}
Hence 
\begin{equation}\label{S2e11}
q=\Phi_{k+1}\in Y_{k+1}:=\{\varphi\in C^2(\us):\,\Delta\varphi=-(k+1)(k+2)\varphi\}.
\end{equation}
 It is well known that $Y_{k+1}$ is $2k+3$ dimensional if $k+1\ge 0$ and $2|k|-3$ dimensional if $k+1<0$.
Then we can solve \eqref{S2e9.4}. Notice that $\lambda(\lambda-1)\neq (\lambda+1)(\lambda+2)$ since $\lambda\in\mathbb{Z}$. Thus we have 
\begin{equation}\label{S2e12}
\vartheta=\frac{k+2}{4k+2}\Phi_{k+1}+\Phi_{k-1},\qquad {\rm where}\,\,\Phi_{k-1}\in Y_{k-1}.
\end{equation}
Now recalling the definition that $d\xi=\omega \Omega_0$ with $\Omega_0$ being the standard volume form on $\us$, and using the equation \eqref{S2e3} (where the right hand side vanishes in this case), and \eqref{S2e9.5}, we conclude that
\begin{equation}\label{S2e13}
\omega=*d\xi=\Phi_k, \end{equation}
and
\begin{equation}\label{S2e14}
\qquad{\rm and}\qquad d^{*}\xi=-\frac{(k+2)(k-1)}{4k+2}\Phi_{k+1}-(k-1)\Phi_{k-1}.
\end{equation}
where $\Phi_k\in Y_k$.

For references below, we recall some property of the spherical harmonic functions. Fix $k\in\Z, k\ge 1$, then 
\begin{equation}\label{S31e3}
Y_k={\rm span}\,\left\{\Phi_k^m(\theta,\phi), \,\,\theta\in[0,2\pi], \,\,\phi\in[0,2\pi],\,\,-k\leq m\leq k\right\},
\end{equation}
where the spherical harmonics $\Phi_k^m(\theta,\phi)$ are given by Legendre polynomials as follows
\begin{equation}\label{S31e4}
\begin{split}
\Phi_k^m(\theta,\phi)&=P_k^m(\cos{\theta}) e^{im\phi},\\
P_k^m&=(-1)^m\frac{1}{2^kk!} \big(1-x^2\big)^{m/2}\frac{d^{k+m}}{dx^{k+m}}\big(x^2-1\big)^k,\quad {\rm for}\,\,0\leq m\leq k,\\
P_k^{-m}&=(-1)^m\frac{(k-m)!}{(k+m)!}P_k^m(x)\,\, {\rm for}\,\,0\leq m\leq k,\quad P_k(x):=P_k^0(x).
\end{split}
\end{equation}
We also define
\begin{equation}\label{S31e5}
P_k^m(x)\equiv 0,\quad {\rm if}\,\,|m|>k.
\end{equation}
We have the following orthogonality property for $0\leq m\leq k$ and $\ell\ge0$,
\begin{equation}\label{S31e6}
\int_{-1}^1P_k^m(x)P_{\ell}^m(x)\,dx=\frac{2\cdot(k+m)!}{(2k+1)\cdot(k-m)!}\mathbf{1}_{k=\ell},
\end{equation}
and the recursive formula for $0\leq m\leq k$, $k\ge1$,
\begin{equation}\label{S31e7}
\big(1-x^2\big)\frac{d}{dx}P_k^m(x)=\frac{1}{2k+1}\Big[(k+1)(k+m)P_{k-1}^m(x)-k(k-m+1)P_{k+1}^m(x)\Big].
\end{equation}
The formulae \eqref{S31e6}-\eqref{S31e7} are standard for Legendre polynomials, and can be verified by direct calculations. We also define for $k\in\Z, k\ge1, -k\leq m\leq k$, the normalized (in $L^2(\us)$ up to an unimportant constant) Legendre polynomials
\begin{equation}\label{S31e8}
\overline{P_k^m}(x):=\frac{1}{N_{k,m}}P_k^m(x), \quad{\rm where}\,\,N_{k,m}:=\left[\int_{-1}^1\big|P_k^m(x)\big|^2\,dx\right]^{1/2},
\end{equation}
and the normalized (in $L^2(\us)$ up to an unimportant constant) spherical harmonics
\begin{equation}\label{S31e9}
\overline{\Phi_k^m}(\theta,\phi):=\overline{P_k^m}(\cos{\theta})e^{im\phi}.
\end{equation}

We summarize the above calculations in the following proposition.
\begin{proposition}\label{S2P1}
Define the space $Y_k, k\in\Z$ as the $k-$th spherical harmonic functions 
\begin{equation}\label{S2e14.1}
Y_k:=\Big\{\varphi\in C^2(\us):\,\Delta\varphi=-k(k+1)\varphi\Big\}.
\end{equation}

Recall the definitions \eqref{S1e11}-\eqref{S1e11.1}. Suppose that $L_0: \mathcal{D}\to X$ is defined as in \eqref{S1e10}. Then the spectrum $\sigma(L_0)$ of $L_0$ is $\mathbb{Z}$. For any $k\in\mathbb{Z}$, the 
corresponding eigenspace $E_k$ with eigenvalue $k$ is given by
\begin{equation}\label{S2e15}
\begin{split}
E_k=\Big\{(\xi,\, \xi',\vartheta,\vartheta^\ast):&\, * d\xi=\Phi_k, \vartheta=\frac{k+2}{4k+2}\Phi_{k+1}+\Phi_{k-1}, \int_{\us}\vartheta=0, \\
&d^\ast\xi=-(k-1)\vartheta, \xi'=-k\xi,\\
&\int_{\us} d\xi=0, -\vartheta^\ast=\Phi_{k+1}+k\vartheta, \\
&{\rm where\,\,}\,\,\Phi_{k-1}\in Y_{k-1}, \Phi_k\in Y_k, \Phi_{k+1}\in Y_{k+1}\Big\}.
\end{split}
\end{equation}
\end{proposition}

Recall that we have the orthogonal decomposition $$L^2(\us)=\oplus_{k\in\mathbb{Z}\cap [0,\infty)} Y_k.$$
For $m\in\{-2,-1,1,2\}$, and any $\phi,\psi\in C^{\infty}(\us)$, we define the inner product
\begin{equation}\label{S2e16}
(\phi,\psi)_{H^{-m}(\us)}:=(\Phi_0,\Psi_0)_{L^2(\us)}+\sum_{k\in\mathbb{Z}\cap[1,\infty)} \big(k(k+1)\big)^{-m} (\Phi_k,\Psi_k)_{L^2(\us)},
\end{equation}
if
\begin{equation}\label{S2e16.1}
\phi=\sum_{k=0}^{\infty}\Phi_k, \quad \psi=\sum_{k=0}^{\infty}\Psi_k, \qquad \Phi_k,\Psi_k\in Y_k, \,\,{\rm for\,\,each}\,\,k\ge0, k\in\mathbb{Z};
\end{equation}
we also define for any $\phi\in C^{\infty}(\us)$
\begin{equation}\label{S2e16.2}
\|\phi\|_{H^{-m}}^2:=(\phi,\phi)_{H^{-m}}. 
\end{equation}
It is clear that the spaces $H^{-m}(\us)$ defined naturally as the completion of smooth functions on the unit sphere with respect to $H^{-m}$ norms are Hilbert spaces. 

We can extend the Hilbert structure to the spaces $X$ and $\mathcal{D}$. 

\begin{definition}\label{definn}
For $\Xi_1:=(\xi_1,\xi_1',\vartheta_1,\vartheta_1^{\ast})\in X$ and $\Xi_2:=(\xi_2,\xi_2',\vartheta_2,\vartheta_2^{\ast})\in X$,  we define the inner product  
\begin{equation}\label{Se22}
\begin{split}
(\Xi_1,\Xi_2)_X:=&(\ast d\xi_1,\ast d\xi_2)_{H^1(\us)}+(d^\ast\xi_1,d^\ast\xi_2)_{H^1(\us)}+(\ast d\xi_1',\ast d\xi_2')_{L^2(\us)}\\
&+(d^\ast\xi_1',d^\ast\xi_2')_{L^2(\us)}
+(\vartheta_1,\vartheta_2)_{H^2(\us)}+(\vartheta^\ast_1,\vartheta_2^\ast)_{H^1(\us)};
\end{split}
\end{equation}
and similarly for $\Xi_1:=(\xi_1,\xi_1',\vartheta_1,\vartheta_1^{\ast})\in \mathcal{D}$ and $\Xi_2:=(\xi_2,\xi_2',\vartheta_2,\vartheta_2^{\ast})\in \mathcal{D}$,  we define the inner product  
\begin{equation}\label{Se22.01}
\begin{split}
(\Xi_1,\Xi_2)_{\mathcal{D}}:=&(\ast d\xi_1,\ast d\xi_2)_{H^2(\us)}+(d^\ast\xi_1,d^\ast\xi_2)_{H^2(\us)}+(\ast d\xi_1',\ast d\xi_2')_{H^1(\us)}\\
&+(d^\ast\xi_1,d^\ast\xi_2)_{H^1(\us)}+(\vartheta_1,\vartheta_2)_{H^3(\us)}+(\vartheta^\ast_1,\vartheta_2^\ast)_{H^2(\us)}.
\end{split}
\end{equation}
\end{definition}

 Although the eigenvalues and eigenspaces of $L_0$ are given explicitly, there exist eigenvectors correspdoning to distinct, sufficiently large eigenvalues of $L_0$ which become increasingly parallel. As a consequence, projections to eigenspaces are not uniformly bounded for large eigenvalues. To address this issue, we use the following decomposition of $X$. We begin with the following definition.
 
 \begin{definition}\label{Def1}For any $\Xi=(\xi,\xi',\vartheta,\vartheta^{\ast})\in X$, we represent $\Xi$ as
\begin{equation}\label{ProP3}
    \Xi=\big\lgroup\ast d\xi, d^\ast\xi, \ast d\xi', d^\ast \xi', \vartheta, \vartheta^\ast\big\rgroup, 
\end{equation}
and define for $k\in\Z\cap[0,\infty), \Phi_k\in Y_k$,
\begin{equation}\label{ProP4}
\begin{split}
    X_{\Phi_k}:=\bigg\{\Xi\in X|\,\, \Xi&=\big\lgroup c_1,c_2,c_3,c_4,c_5,c_6\big\rgroup\Phi_k\\
    &:=\big\lgroup c_1\Phi_k,c_2\Phi_k,c_3\Phi_k,c_4\Phi_k,c_5\Phi_k,c_6\Phi_k\big\rgroup,\,\,{\rm where}\,\,c_j\in\mathbb{C}\bigg\},
    \end{split}
\end{equation}
and define for any orthonormal basis $\Sigma$ for $Y_k$, \begin{equation}\label{ProP4.5}
    X_k:=\oplus_{\Phi_k\in \Sigma} X_{\Phi_k}.
\end{equation}
Finally, as $Y_{-k-1}=Y_k$ for $k\in\Z\cap[0,\infty)$, we set
\begin{equation}\label{ProP4.51}
    X_{-k-1}:=X_k.
\end{equation}
\end{definition}
It follows from \eqref{S2e15} that for $k\in\Z$, 
  \begin{equation}\label{ProP4.1}
    E_k\subseteq X_{k-1}\oplus X_{k}  \oplus X_{k+1}.
  \end{equation}

  It is easy to check that $L_0: X_{\Phi_k}\to X_{\Phi_k}$ for any $\Phi_k\in Y_k$, and $X$ can be decomposed as the direct sum of $X_{\Phi_k}$ where $k$ ranges over $\in \Z\cap[0,\infty)$ and $\Phi_k$ ranges over an orthonormal basis of $Y_k$.  

\begin{remark}\label{remark4.1}
We note that $E_0$ consists of vectors $\Xi$ of the form $$\big\lgroup0,1,0,0,1,-1\big\rgroup\Phi_1$$ where $\Phi_1\in Y_1$; and that $E_1$ consists of vectors $\Xi$ of the form $$\big\lgroup1,0,-1,0,0,0\big\rgroup\Phi_1+\big\lgroup0,0,0,0,1/2,-3/2\big\rgroup\Phi_2$$ where $\Phi_1\in Y_1, \Phi_2\in Y_2$. In particular, ${\rm dim}\,E_0=3$ and ${\rm dim}\,E_1=8$. 
\end{remark}

We have the following property for the resolvent of $L_0$. 
\begin{proposition}\label{ProP1}
    For any $\lambda\in\mathbb{C}\backslash\Z$ with the property that 
    \begin{equation}
        \min_{k\in\Z}\{|\lambda-k|\}\ge 1/9,
    \end{equation}
    the resolvent of $L_0$ satisfies the bound
    \begin{equation}\label{ProP2}
        \|(\lambda-L_0)^{-1}\|_{X\to X}\lesssim \frac{1}{1+|\Im \lambda|}+\frac{|\Re \lambda|+1}{(|\Im \lambda|+1)^2}. 
    \end{equation}

\end{proposition}

\begin{proof}

We begin with the representation of $\Xi\in X$ as linear combinations of eigenvectors of $L_0$. 

We first consider the main case when $k\ge1$. Let $\Phi_k\in Y_k$. Assume that $c_i\in\mathbb{C}, i\in\Z\cap[1,6]$.  Notice that
\begin{equation}\label{ProP5}
   \big\lgroup c_1, 0, c_3, 0, 0,0\big\rgroup\Phi_k \in{\rm span}\,\big\{\Xi_1, \Xi_2\big\},
\end{equation}
where the vectors $\Xi_1,\Xi_2$ are given by
\begin{equation}\label{ProP6}
    \begin{split}
        \Xi_1:=&\big\lgroup\frac{k+1}{2k+1},0,-\frac{k(k+1)}{2k+1},0,0,0\big\rgroup\Phi_k\in E_k,\\
        \Xi_2:=&\big\lgroup\frac{k}{2k+1},0,\frac{k(k+1)}{2k+1},0,0,0\big\rgroup\Phi_k\in E_{-k-1}. 
    \end{split}
\end{equation}

We now turn to the vector $\big\lgroup0,c_2,0,c_4,c_5,c_6\big\rgroup\Phi_k$. Notice that 
\begin{equation}\label{ProP6.1}
\big\lgroup0,c_2,0,c_4,c_5,c_6\big\rgroup\Phi_k\in{\rm span}\,\big\{\Xi_3,\Xi_4,\Xi_5,\Xi_6\big\},
\end{equation}
where the vectors $\Xi_j,j\in\{3,4,5,6\}$ are given by
\begin{equation}\label{Prop7}
    \begin{split}
    &  \Xi_3:=\big\lgroup0,-\frac{(k-2)(k+1)}{4k-2} , 0,  \frac{(k-1)(k-2)(k+1)}{4k-2} ,\frac{k+1}{4k-2} \\
    &\qquad\qquad\qquad\qquad\qquad\qquad\qquad,-1-\frac{(k-1)(k+1)}{4k-2}\big\rgroup\Phi_k\in E_{k-1},\\
&\Xi_4:=\big\lgroup 0,-k,0,  k(k+1),1, -k-1\big\rgroup\Phi_k\in E_{k+1},\\
&\Xi_5:=\big\lgroup0, \frac{(k+3)k}{4k+6},0, \frac{k(k+2)(k+3)}{4k+6}, \frac{k}{4k+6},-1+\frac{k(k+2)}{4k+6}\big\rgroup\Phi_k\in E_{-k-2},\\
&\Xi_6:=\big\lgroup0,k+1,0,k(k+1), 1,k\big\rgroup\Phi_k\in E_{-k},
    \end{split}
\end{equation}

By direct calculation, the matrix
\begin{equation}\label{S2e37}
\mathcal{M}_k:=\begin{pmatrix}
  -\frac{(k-2)(k+1)}{4k-2} &  \frac{(k-1)(k-2)(k+1)}{4k-2} & \frac{k+1}{4k-2} & -1-\frac{(k-1)(k+1)}{4k-2}\\
-k&  k(k+1)&1& -k-1\\
\frac{(k+3)k}{4k+6}& \frac{k(k+2)(k+3)}{4k+6}& \frac{k}{4k+6}&-1+\frac{k(k+2)}{4k+6}\\
k+1&k(k+1)& 1& k
\end{pmatrix}
\end{equation}
is non-degenerate for $k\in\Z\cap[0,\infty)$. Indeed, we have
$${\rm det}\,\mathcal{M}=-(2k+1)^2\neq0, \quad {\rm for}\,\,k\in\Z\cap[0,\infty).$$

If $k=0$, the only vector $\big\lgroup c_1,c_2,c_3,c_4,c_5,c_6\big\rgroup\Phi_0\in X$ is of the form $$\big\lgroup0,0,0,0,0,1\big\rgroup\Phi_0,$$ where $\Phi_0\in Y_0$. It is clear that this vector belongs to the eigenspace $E_{-2}$ of $L_0$. 

By direct computation, it is easy to check that (i) $\Xi_1,\Xi_2$ are orthogonal to $\Xi_3,\Xi_4,\Xi_5,\Xi_6$; (ii) $\Xi_1$ and $\Xi_2$ form an angle $\theta\in[\pi/2-\delta_0,\pi/2+\delta_0]$ for some $\delta_0\in(0,\pi/2)$ independent of $k\in\Z\cap[0,\infty)$; (iii) the vectors $\Xi_3,\Xi_4$ (and similarly $\Xi_5,\Xi_6$) are nearly parallel for large $k\gg1$, and they form an angle of size $\theta_k\approx \frac{1}{k+1}$; (iv) the vectors $\Xi_3$ and $\Xi_5$ form an angle $\theta'_k\in[\pi/2-\delta_1,\pi/2+\delta_1]$ for some $\delta_1\in(0,\pi/2)$ independent of $k\in\Z\cap[1,\infty)$. 

Therefore in view of the relations \eqref{ProP5}, \eqref{ProP6.1}, we conclude that (with $S_k:=\{-k-2,-k,k-1,k,k+1\}$) 
\begin{equation}\label{ProP8}
    \begin{split}
        &\|(\lambda-L_0)^{-1}\|_{X\to X}\\
        &\lesssim \sup_{ k\in\Z\cap[0,\infty), \Phi_k\in Y_k}\|(\lambda-L_0)^{-1}\|_{X_{\Phi_k}\to X}\\
        &\lesssim \sup_{k\in\Z\cap[0,\infty), \,\,j\in S_k}\bigg[\frac{k+1}{|\lambda-k+1||\lambda-k-1|}+\frac{k+1}{|\lambda+k+2||\lambda+k|}+\frac{1}{|\lambda-j|}\bigg].  
    \end{split}
\end{equation}
The desired bound \eqref{ProP2} then follows.
\end{proof}

We now turn to the property of the semigroups generated by $L_0$. Since the spectrum of $L_0$ is $\Z$, we need to consider the subspaces of $X$ corresponding to the positive and negative part of the spectrum, respectively. For applications below we introduce the following projection operators.
\begin{definition}\label{DEF2}
   For $S\subseteq \Z$, we define the linear operator $P_S$ as the spectral projection operator from $X$ to the subspace of $X$ corresponding to the spectrum $S$ of $L_0$. 
   
   More precisely, by linearity, it suffices to define $P_S\Xi$ for $\Xi\in X_{\Phi_k}$ for $k\in\Z\cap[0,\infty)$ and $\Phi_k\in Y_k$. Suppose that $\Xi\in X_{\Phi_k}$. Recalling \eqref{ProP6} and \eqref{Prop7}, we can represent $\Xi$ as 
   \begin{equation}\label{ProP9}
       \Xi=\sum_{j=1}^6c_j\Xi_j, \quad{\rm where}\,\,c_j\in\mathbb{C}. 
   \end{equation}
   Define the function $f:\Z\cap[1,6]\to \{-k-2,-k-1,-k,k-1,k,k+1\}$ such that $\Xi_j\in E_{f(j)}$ for $j\in\Z\cap[1,6]$, and then we can set
   \begin{equation}\label{ProP10}
       P_S\Xi:=\Sigma_{\{j\in\Z\cap[1,6]:\,f(j)\in S\}}\,c_j\Xi_j.
   \end{equation}
\end{definition}
We note that the operator $P_S$ may not be uniformly bounded for all $S\subseteq \Z$. However, we have the following lemma. 
\begin{lemma}\label{ProP11}
    Denote for $m\in\Z$, $P_{m}:=P_{\{m\}}$, $P_{\ge m}:=P_{\Z\cap[m,\infty)}$ and $P_{\leq m}:=P_{\Z\cap(-\infty,m]}$. Then we have the bounds for any $\Xi\in X$,
    \begin{equation}\label{ProP12}
        \|P_0\Xi\|_{X}+ \|P_1\Xi\|_{X}+\|P_{\ge 2}\Xi\|_{X}+\|P_{\leq -1}\Xi\|_{X}\lesssim\|\Xi\|_X. 
    \end{equation}
    
Moreover, the semigroups $e^{L_0s}P_{\ge2}$ for $s\leq0$ and $e^{L_0s}P_{\leq -1}$ for $s\ge0$ satisfy the bounds for any $\Xi\in X$,
\begin{equation}\label{ProP13}
    \begin{split}
        \|e^{L_0s}P_{\ge2}\Xi\|_{X}&\lesssim e^{2s}\|\Xi\|_X\quad{\rm for}\,\,s\leq0;\\
         \|e^{L_0s}P_{\leq -1}\Xi\|_{X}&\lesssim e^{-s}\|\Xi\|_X\quad{\rm for}\,\,s\ge0.
    \end{split}
\end{equation}
    
\end{lemma}

\begin{proof}
    It suffices to prove the desired bounds \eqref{ProP12}-\eqref{ProP13} for $\Xi\in Y_{\Phi_k}$, where $k\in\Z\cap[0,\infty), \Phi_k\in Y_k$. Recall that $X_{\Phi_k}$ is invariant under under $L_0$. By expressing $\Xi$ as the linear combination of $\Xi_j, j\in\Z\cap[1,6]$ the proof then follows from a finite dimensional analysis on $X_{\Phi_k}$. 
\end{proof}

%
\section{Spectral properties of $L$ and proof of Theorem \ref{MTh1}}
In this section we study the spectral property of $L$ and obtain bounds on its associated semigroups. As an application, we give the proof of our main Theorem \ref{MTh1}. 

Define the operator $K: \mathcal{D}\to X$ as
\begin{equation}\label{S31e1}
K\begin{pmatrix}
&\xi\\
&\xi'\\
&\vartheta\\
&\vartheta^\ast
\end{pmatrix}=\begin{pmatrix}
&0\\
&-{\bf V}^b\cdot\nabla\xi-\xi\cdot\nabla {\bf V}^b-F^b\xi'\\
&0\\
&-{\bf V}^b\cdot\nabla\vartheta-\xi\cdot\nabla F^b+2\xi\cdot {\bf V}^b+2\vartheta F^b-\vartheta^\ast F^b-qF^b
\end{pmatrix}
\end{equation}
where ${\rm div}\,\xi+\vartheta+\vartheta^\ast+q=0$. We note that $L=L_0+K$.
and that $K$ satisfies the bound for any $\Xi\in X$,
\begin{equation}\label{S31e1.001}
    \|K\Xi\|_{X}\lesssim|b|\|\Xi\|_{X}. 
\end{equation}

We have the following property for the resolvent of $L$. 
\begin{proposition}\label{ProPP1}
  There exists $\kappa\in(0,1)$ sufficiently small such that for $b\in\R^3$ with $|b|<\kappa$ the following statement holds. For any $\lambda\in\mathbb{C}$ satisfying 
    \begin{equation}\label{ProPP1.1}
        \min_{k\in\Z}\{|\lambda-k|\}\ge 1/9, \quad 100(1+|\Re \lambda|)^{7/8}\ge|\Im \lambda|\ge \frac{1}{100}(1+|\Re \lambda|)^{7/8},
    \end{equation}
 we have $\lambda\not\in\sigma(L)$ and the bounds 
    \begin{equation}\label{ProPP2}
        \|(\lambda-L)^{-1}\|_{X\to X}\lesssim (1+|\lambda|)^{-3/4}. 
    \end{equation}
    
\end{proposition}
\begin{proof}
    Using the (formal) resolvent identity
    \begin{equation}\label{ProPP3}
        (\lambda-L)^{-1}=\Big[I-(\lambda-L_0)^{-1}K\Big]^{-1}(\lambda-L_0)^{-1},
    \end{equation}
    The desired bounds \eqref{ProPP2} then follow from the bounds \eqref{ProP2} and the inequality 
    \begin{equation}\label{ProPP4}
        \|K\|_{X\to X}\lesssim |b|. 
    \end{equation}
\end{proof}

We now turn to the semigroups generated by $L$. In contrast to the case of $L_0$, these semigroups are no longer explicitly given and needed to be defined more abstractly using spectral representation formulae. 

Define the contours $\Gamma_+$ and $\Gamma_-$ as
\begin{equation}\label{ProPP5}
    \begin{split}
        \Gamma_+:=&\Big\{\lambda\in\mathbb{C}:\,\Re\lambda\ge3/2, \,\,|\Im\lambda|=(\Re\lambda-3/2)^{7/8}\Big\},\\
        \Gamma_-:=&\Big\{\lambda\in\mathbb{C}:\,\Re\lambda\leq -1/2,\,\, |\Im\lambda|=(-\Re\lambda-1/2)^{7/8}\Big\}. 
    \end{split}
\end{equation}
Then we can define the operator for $\Xi\in X$,
\begin{equation}\label{ProPP6}
    \begin{split}
        e^{Ls}P^\ast_+\,\Xi:=&\frac{1}{2\pi i}\int_{\Gamma_+}e^{\lambda s}(\lambda-L)^{-1}\Xi\,d\lambda,\quad{\rm for}\,\,s\leq0;\\
       e^{Ls} P^\ast_j\,\Xi:=&\frac{1}{2\pi i}\int_{|\lambda-j|=1/2}e^{\lambda s}(\lambda-L)^{-1}\Xi \,d\lambda, \quad j\in\{0,1\},\quad {\rm for}\,\,s\in\R;
    \end{split}
\end{equation}
and define for $s\ge0$ and $\Xi\in X$,
\begin{equation}\label{ProPP7}
    \begin{split}
        e^{Ls}P^\ast_-\,\Xi:=&\frac{1}{2\pi i}\int_{\Gamma_-}e^{\lambda s}(\lambda-L)^{-1}\Xi\,d\lambda.
    \end{split}
\end{equation}

We have the following bounds on these semigroups. 
\begin{proposition}\label{ProPP8}
    For $s\leq 0$ and $\Xi\in X$ we have the bound 
    \begin{equation}\label{ProPP9}
        \|e^{Ls}P^\ast_+\,\Xi\|_{X}\lesssim e^{7s/4}\|\Xi\|_X,
    \end{equation}
    and for $s\ge 0$ and $\Xi\in X$ we have 
    \begin{equation}\label{ProPP10}
        \|e^{Ls}P^\ast_-\,\Xi\|_{X}\lesssim e^{-3s/4}\|\Xi\|_X.
    \end{equation}
    In addition, for any $\sigma,\sigma'\in\{+,-,1,2\}$ denoting $\delta(\sigma\sigma')$ as the standard Kronecker's delta symbol, we have the identities
    \begin{equation}\label{ProPP11}
        P^\ast_\sigma P^\ast_{\sigma'}=\delta(\sigma\sigma')P^\ast_{\sigma},\quad I=P^\ast_++P^\ast_-+P^\ast_0+P^\ast_1.
    \end{equation}
\end{proposition}

\begin{proof}
   We use the resovlent identity \eqref{ProPP3} which can be expanded in power series since $K$ has a small norm
    \begin{equation}\label{ProPP12}
        (\lambda-L)^{-1}=\sum_{m=0}^{\infty}\Big[(\lambda-L_0)^{-1}K\Big]^m(\lambda-L_0)^{-1}. 
    \end{equation}
    We note that in the formulae \eqref{ProPP6}-\eqref{ProPP7}, the first term of \eqref{ProPP12} with $m=0$ corresponds to semigroups associated with $L_0$ for which the desired bounds \eqref{ProPP9}-\eqref{ProPP10} follow from \eqref{ProP13}. For the contributions from terms with $m\ge1$, we note that 
    \begin{equation}\label{ProPP13}
        \left\|\Big[(\lambda-L_0)^{-1}K\Big]^m(\lambda-L_0)^{-1}\right\|_{X\to X}\leq (C|b|)^m(1+|\lambda|)^{-3/4-3m/4},
    \end{equation}
    and the desired bounds \eqref{ProPP9}-\eqref{ProPP10} then follow from a contour shift. 
    
\end{proof}

It remains to treat the semigroups $e^{Ls}P^\ast_0$ with $s\in\R$ and $e^{Ls}P^\ast_1$ with $s\leq0$. 
\begin{definition}\label{defsep}
  Fix $\gamma\in(0,1)$.  Let $\Sigma\subset X$ be a set of $J$ vectors. $\Sigma$ is called $\gamma$-separated if for any  any $c_j\in\mathbb{C}$  and $\Xi_j\in \Sigma$ with $1\leq j\leq J\in \Z\cap[1,\infty)$, we have
    \begin{equation}\label{mprop6}
    \gamma\,\sum_{j=1}^J|c_j|\|\Xi_j\|_X  \leq  \Big\|\sum_{j=1}^Jc_j\Xi_j\Big\|_X\leq  \sum_{j=1}^J|c_j|\|\Xi_j\|_X.
    \end{equation}
    \end{definition}

The main properties we need are summarized as follows. 
\begin{proposition}\label{mprop1}
    There exist small constants $\kappa, \gamma\in(0,1)$ such that the following statement holds for $b\in\R^3$ with $0\leq|b|<\kappa$. Recall the relation $|b|=f(\epsilon)$ from \eqref{Int3.3}. 
    
    (i) $\lambda=0$ is an eigenvalue of $L$ with the corresponding eigenspace $P^\ast_0X$ satisfying ${\rm dim}\,P^\ast_0 X=3$;

    (ii) There exist subspaces ($E_b=E_b^1\oplus E_b^2\oplus E_b^3$)
    \begin{equation}\label{mprop1.01}
    \begin{split}
    E_b:=&{\rm span}\,\Big\{\Xi_{b,j}:\,j\in\Z\cap[1,8]\Big\},\quad E_b^1:={\rm span}\,\Big\{\Xi_{b,j}:\,j\in\Z\cap[1,4]\Big\},\\
    E_b^2:=&{\rm span}\,\Big\{\Xi_{b,j}:\,j\in\Z\cap[5,6]\Big\}, \quad F_b^3:={\rm span}\,\Big\{\Xi_{b,j}:\,j\in\Z\cap[7,8]\Big\},
    \end{split}
    \end{equation}
    which are invariant under $L$. The set 
    \begin{equation}\label{mprop1.02}
    \Sigma:=\Big\{\Xi_{b,j}:\,j\in\Z\cap[1,8]\Big\}
    \end{equation}
    is $\gamma$-separated, and the $8\times 8$ matrix given by the restriction $L|_{E_b}$ of $L$ to $E_b$ with respect to $\Sigma$ has several eigenvalues. 
    
    (iii) With respect to the base $\Sigma$, we have the representations
    \begin{enumerate}
        \item 
    
    \begin{equation}\label{mprop3}
        L|_{E^1_b}\sim I=\begin{pmatrix}
            1&0&0&0\\
            0&1&0&0\\
            0&0&1&0\\
            0&0&0&1
        \end{pmatrix};
        \end{equation}

        \item 

        \begin{equation}\label{mprop4}
             L|_{E^2_b}\sim \begin{pmatrix}
           \alpha_2&-\beta_2\\
           \beta_2&\alpha_2
        \end{pmatrix},
        \end{equation}
    where $\alpha_2=1+\frac{1}{15}\epsilon^2+O(\epsilon^3)$, $\beta_2=O(\epsilon^3)$;

    \item 
      \begin{equation}\label{mprop5}
             L|_{E^3_b}\sim \begin{pmatrix}
           \alpha_3&-\beta_3\\
           \beta_3&\alpha_3
        \end{pmatrix},
        \end{equation}
    where $\alpha_3=1+\frac{4}{15}\epsilon^2+O(\epsilon^3)$, $\beta_3=O(\epsilon^3)$.
    \end{enumerate}

    \item

\end{proposition}

\begin{remark}\label{rem1}
    There are three obvious vectors in $E^1_b$: those generated by $\partial_1U^b, \partial_2U^b$ and $\partial_3U^b$ through \eqref{S1e4}, where $U^b$ is the corresponding Landau solution to \eqref{Int3}. 
    
    The fourth vector is more nontrivial and is given by a ``pure-swirl" field of the form $C\frac{\sin\theta}{(1-\epsilon \cos\theta)^2}e_\phi$, assuming (without loss of generality) that $b=f(\epsilon) \begin{pmatrix}
        0\\0\\1
    \end{pmatrix}$ as in \eqref{Int3.3} and using the spherical coordinate \eqref{Int3.1}-\eqref{Int3.10}. See section \ref{sec:m0} for details.

    In particular, if we restrict our attention to the axi-symmetric fields (not necessarily with zero swirl), the original eigenspace $E_1^{\rm axisymm}$ is two dimensional, and the corresponding eigenvalue $\lambda=1$ does not move. It stays $\lambda_b=1$ for all $|b|<1$. (The space $E_b^{\rm axisymm}$ can move with $b$.)

    We note that there is a slight abuse of notation in the sense that as $b\to0$, $E_b$ converges to $E_1$ rather than $E_0$. This choice is motivated by the perturbation argument below. 
\end{remark}

\begin{remark}\label{rem2}
The proof of (i) in Proposition \ref{mprop1} is relatively easy. Let $U^b$ be the Landau solution to \eqref{Int3} which is $-1$-homogeneous. For any vector $c=(c_1,c_2,c_3)\in\R^3$, $\varphi:=c\cdot\nabla_bU^b$ satisfies for $x\in\R^3$,
\begin{equation}\label{comp1.1}
\begin{split}
    -\Delta\varphi+\varphi\cdot \nabla U^b+U^b\cdot\nabla\varphi+\nabla p&=c\,\delta(x),\\
    {\rm div}\,\varphi=0. 
    \end{split}
\end{equation}
We conclude that $c\cdot\nabla_bU^b\not\equiv0$ unless $c=0\in\R^3$, and consequently $\partial_{b_1}U^b, \partial_{b_2}U^b$ and $\partial_{b_3}U^b$ are linearly independent. We can form three linearly independent vectors $\Xi^\ast_{b,1}, \Xi^\ast_{b,2}, \Xi^\ast_{b,3}$ in $X$ using the procedure \eqref{Pe3}-\eqref{Pe6}. Since $\partial_{b_j}U^b(x)$ with $j\in\{1,2,3\}$ are $-2$-homogeneous, these vectors belong to $P_0^\ast X$ which is three dimensional. Therefore $P_0^\ast X={\rm span}\Big\{\Xi^\ast_{b,1}, \Xi^\ast_{b,2}, \Xi^\ast_{b,3}\Big\}$. 
\end{remark}

The proof of (ii)-(iii) in Proposition \ref{mprop1} is the heart of our analysis and requires relatively precise computations. We will give the detailed proof in sections \ref{sec:comp1} - \ref{sec:comp3} below, and provide more details on the invariant spaces $E^1_b, E^2_b, E^3_b$. 

As a corollary to Proposition \ref{mprop1}, we have the following conclusions. 
\begin{proposition}\label{mprop7}
    Under the same assumptions and notations as in Proposition \ref{mprop1}, we have for $s\leq0$,
    \begin{equation}\label{mprop8}
    \begin{split}
        &e^{Ls}P^\ast_1|_{E^1_b}\sim e^{s}I,\\
        &e^{Ls}P^\ast_1|_{E^2_b}\sim e^{\alpha_2s}\begin{pmatrix}
            \cos (\beta_2s)&-\sin(\beta_2s)\\
            \sin(\beta_2s)&\cos(\beta_2s)
        \end{pmatrix}, \\
        &e^{Ls}P^\ast_1|_{E^3_b}\sim e^{\alpha_3s}\begin{pmatrix}
            \cos (\beta_3s)&-\sin(\beta_3s)\\
            \sin(\beta_3s)&\cos(\beta_3s)
        \end{pmatrix}.
        \end{split}
    \end{equation}
\end{proposition}

\begin{proof}
    \eqref{mprop8} follows directly from the matrix representations \eqref{mprop3}-\eqref{mprop5}. 
\end{proof}

We can now give the proof of our main Theorem \ref{MTh1}. 
\begin{proof}[Proof of Theorem \ref{MTh1}]
For any vector $\Xi\in X$ with $\Xi=\begin{pmatrix}
\xi\\
\xi'\\
\vartheta\\
\vartheta^\ast
\end{pmatrix}\in X$, we define
\begin{equation}\label{S5e17.1}
F(\Xi)=\begin{pmatrix}
0\\
-\xi\cdot\nabla\xi-\vartheta\xi'\\
0\\
-\xi\cdot\nabla\vartheta+\vartheta^2+\vartheta(\vartheta^\ast+q)\end{pmatrix},
\end{equation}
where we assume ${\rm div}\,\xi+\vartheta+\vartheta^\ast+q=0$, which is the usual divergence free condition.

It is simple to verify that we have the bounds for $\Xi\in X$,
\begin{equation}\label{S6e2}
\big\|F\big(\Xi\big)\big\|_X\lesssim \big\|\Xi\big\|_X^2.
\end{equation}

By Proposition \ref{Pe}, the perturbed steady Navier Stokes equation \eqref{Pe1} in $\R^3\backslash B_1$ can be reformulated in terms of $\Xi:(-\infty,0]\to X$, 
\begin{equation}\label{S6e3}
\partial_s\Xi(s)=L\Xi(s)+F(\Xi(s)),\qquad{\rm for}\,\,s\in(-\infty,0].
\end{equation}
The argument below is similar to the construction of unstable manifolds in dynamical systems. Indeed, by the result in \cite{SverakKorolev}, we know that 
\begin{equation}\label{S6e4}
\big\|\Xi(s)\big\|_X\leq C\delta_0\,e^{7s/8}. 
\end{equation}

Using the bounds \eqref{S6e4}, Proposition \ref{mprop1} and Proposition \ref{mprop7}, we have for $s\in(-\infty,0]$,
\begin{equation}\label{S6e11}
\begin{split} 
\Xi(s)&=P^\ast_-\Xi(s)+P^\ast_+\Xi(s)+P^\ast_0\Xi(s)+P^\ast_1\Xi(s),\\
P^\ast_-\Xi(s)&=\int_{-\infty}^se^{L(s-\tau)}P^\ast_-F(\Xi(\tau))\,d\tau, \\
P^\ast_+\Xi(s)&=e^{Ls}P^\ast_+\Xi(0)-\int_s^0e^{L(s-\tau)}P^\ast_+F(\Xi(\tau))\,d\tau, \\
P^\ast_0\Xi(s)&=\int_{-\infty}^sP^\ast_0F(\Xi(\tau))\,d\tau,\\
P^\ast_1\Xi(s)&=e^{Ls}P^\ast_1\Xi(0)-\int_s^0e^{L(s-\tau)}P^\ast_1F(\Xi(\tau))\,d\tau. 
\end{split}
\end{equation}
Note that from the bounds  \eqref{S6e2} and \eqref{S6e4}, equations \eqref{S6e11}, and Propositions \ref{mprop1} and \ref{mprop7}, 
\begin{equation}\label{S6e13}
   \Big\|\Xi(s)-e^{Ls}\Big[P^\ast_1\Xi(0)-\int_{-\infty}^0e^{-L\tau}P_1^\ast F(\Xi(\tau))\,d\tau\Big]\Big\|_X\lesssim \delta_0^2\,e^{7s/4},
\end{equation}
from which our main theorem follows. 
\end{proof}

\section{Computation of the perturbed eigenvalues (I): preliminaries}\label{sec:comp1}
In this section we outline the general plan for computing the eigenvalues of $L$ that are perturbations of the eigenvalues $\lambda=1$ for $L_0$. Throughout the section we assume without loss of generality that $b=f(\epsilon) \begin{pmatrix}
    0\\0\\1
\end{pmatrix}$ as in \eqref{Int3.2}-\eqref{Int3.3} for some $0<\epsilon\ll1$.

\subsection{General formulae for linear operator perturbation theory}

Let us consider the following abstract situation. The notations in this section are independent of those in all other sections. 

Assume that $L_0$ is a linear operator on a Banach space $X$ and we have a decomposition
\begin{equation}\label{GFL1}
    X=E\oplus Y
\end{equation}
where $E$ is finite dimensional, $Y$ is closed and both $E, Y$ are invariant under $L_0$. We will assume that
\begin{equation}\label{GFL2}
    L_0|_E=I_E\quad{\rm and}\quad L_0|_Y=\Lambda,
\end{equation}
where $I-\Lambda$ is invertible on $Y$. 


We assume that $L_0$ is perturbed by a ``small" operator $K$ and would like to compute what happens to the eigenvalue $\lambda=1$ of $L_0$ and its eigenspace $E$. 

Let us write vectors on $X$ as $\begin{pmatrix}
    u\\v
\end{pmatrix}$, with $u\in E$ and $v\in Y$. We denote the projections
\begin{equation}\label{GFL3}
    \begin{pmatrix}
        u\\v
    \end{pmatrix}\to \begin{pmatrix}
        u\\0
    \end{pmatrix}, \quad \begin{pmatrix}
        u\\v
    \end{pmatrix}\to \begin{pmatrix}
        0\\v
    \end{pmatrix}
\end{equation}
by $P$ and $Q$, respectively. Let us write
\begin{equation}\label{GFL4}
    K=\begin{pmatrix}
        A&B\\
        C&D
    \end{pmatrix},
\end{equation}
where $A:E\to E$ is given by $Au=PKu$, $B:Y\to E$ is given by $Bv=PKv$, $C:E\to Y$ is given by $Cu=QKu$ and $D:Y\to Y$ is given by $Dv=QKv$. 

Under the perturbation $L_0\to L_0+K$ the invariant space 
$E\sim \left\{\begin{pmatrix}
    u\\0
\end{pmatrix}\right\}$ of $L_0$ is expected to deform to an invariant space of $L=L_0+K$ of the form
\begin{equation}\label{GFL5}
    F\sim\left\{\begin{pmatrix}
        u\\ Mu
    \end{pmatrix}, \,u\in E\right\}.
\end{equation}
The invariance of $F$ under $L$ is expressed by 
\begin{equation}\label{GFL6}
    L\begin{pmatrix}
        u\\ Mu
    \end{pmatrix}=(L_0+K)\begin{pmatrix}
        u\\ Mu
    \end{pmatrix}=\begin{pmatrix}
        u+Au+BMu\\
        \Lambda Mu+Cu+DMu
    \end{pmatrix}\in F,
\end{equation}
which is the same as 
\begin{equation}\label{GFL7}
 Mu+MAu+MBMu=\Lambda Mu+Cu+DMu, \,\,\forall u\in E,
\end{equation}
or equivalently
\begin{equation}\label{GFL8}
  M=(I-\Lambda)^{-1}C+(I-\Lambda)^{-1}DM-(I-\Lambda)^{-1}MA-(I-\Lambda)^{-1}MBM.  
\end{equation}

We can think of this equation as an equation for a ``fixed point" $M$":
\begin{equation}\label{GFL9}
    M=\Phi(M)
\end{equation}
for the map $\Phi$ given by the right-hand side of \eqref{GFL8}, and iterate $\Phi$ to the the fixed point at the limit of the sequence
\begin{equation}\label{GFL10}
    \Phi(M), \quad\Phi(\Phi(M)),\quad\Phi(\Phi(\Phi(M))), \quad\dots
\end{equation}

When we think of $K$ as ``order $\epsilon$" (writing it as $K=\epsilon \widetilde{K}$), then our solution can be written as $\epsilon C_1+\epsilon^2C_2+\epsilon^3C_3+\dots$, where each $C_i$ is a finite (algebraic) expression in $(I-\Lambda)^{-1}$, $A, B, C, D$. 

The finite dimensional space $F$ will be isomorphic to $E$ through the map  $u\to \begin{pmatrix}
    u\\ Mu
\end{pmatrix}$, and the map $L|_F$ will be isomorphic to the map
\begin{equation}\label{GFL11}
    u\to u+Au+BMu.
\end{equation}
In particular, the eigenvalues of $L|_F$ will be given by the matrix of this map, which we will write as 
\begin{equation}\label{GFL12}
    I+A+BM.
\end{equation}

Our calculation will aim to obtain the expression with an error of $O(\epsilon^3)$. For this level of approximation our matrix is 
\begin{equation}\label{GFL13}
    I+A+B(1-\Lambda)^{-1}C+{\rm error\,\,of\,\,order\,\,}O(\epsilon^3). 
\end{equation}

We summarize our discussion above in the following proposition.

\begin{proposition}\label{GFL13.1}
 Fix $\kappa>1$. Assume that $L_0:\mathcal{D}\subseteq X\to X$ is a closed linear operator  on a Banach space $X$ with non-empty resolvent, and we have a decomposition
\begin{equation}\label{GFL13.2}
    X=E\oplus Y
\end{equation}
where $E$ is finite dimensional, $Y$ is closed and both $E, Y\cap\mathcal{D}$ are invariant under $L_0$. Assume that
\begin{equation}\label{GFL12.3}
    L_0|_E=id_E\quad{\rm and}\quad L_0|_{Y \cap\mathcal{D}}=\Lambda,
\end{equation}
where $id_E$ is the identity map on $E$ and $I-\Lambda$ is boundedly invertible on $Y\cap\mathcal{D}$ with
\begin{equation}\label{GFL12.4}
    \|(I-\Lambda)^{-1}\|_{Y\to Y}\leq\kappa.  
\end{equation}
Write vectors on $X$ as $\begin{pmatrix}
    u\\v
\end{pmatrix}$, with $u\in E$ and $v\in Y$. Let $P,Q$ be the projection operators to $E, Y$ respectively, satisfying the bounds that for any $\Xi\in X$,
\begin{equation}\label{GFL12.5}
\kappa^{-1}\big(\|P\Xi\|_X+\|Q\Xi\|_X\big)\leq\|\Xi\|_X\leq \kappa\big(\|P\Xi\|_X+\|Q\Xi\|_X\big). 
\end{equation}
Suppose that $K$ is a bounded linear operator from $X$ to $X$ given by
\begin{equation}\label{GFL12.6}
    K\begin{pmatrix}
        u\\v
    \end{pmatrix}=\begin{pmatrix}
        A&B\\
        C&D
    \end{pmatrix}\begin{pmatrix}
        u\\v
    \end{pmatrix}=\begin{pmatrix}
        Au+Bv\\Cu+Dv
    \end{pmatrix}
\end{equation}
for certain bounded linear operators $A:E\to E$, $B:Y\to E$, $C: E\to Y$ and $D: Y\to Y$. Moreover for some $\varepsilon\in(0,1)$, we have
\begin{equation}\label{GFL12.7}
    \|K\Xi\|_X\leq\varepsilon \|K\Xi\|_X,\quad\forall \Xi\in X. 
\end{equation}
If $\varepsilon>0$ is sufficiently small depending on $\kappa$, then there exists a bounded linear operator $M:E\to Y$, such that
\begin{equation}
    \|M-(I-\Lambda)^{-1}C\|_{E\to Y}\lesssim_\gamma \varepsilon^2
\end{equation}
and the space
\begin{equation}\label{GFL12.8}
    F:=\left\{\begin{pmatrix}
        u\\Mu
    \end{pmatrix}|\,\,u\in E\right\}
\end{equation}
is invariant under $L:=L_0+K$. In addition, through the isormorhism $u\in E\to \begin{pmatrix}
    u\\M u
\end{pmatrix}\in F$, $L|_F$ is represented by the map 
\begin{equation}\label{GFL12.9}
    u\to u+Au+BMu,
\end{equation}
that is for any $\begin{pmatrix}
    u\\M u
\end{pmatrix}\in F$, 
\begin{equation}\label{GFL12.10}
   L\begin{pmatrix}
    u\\M u
\end{pmatrix}=\begin{pmatrix}
 u+Au+BMu\\M(u+Au+BMu)   
\end{pmatrix}.
\end{equation}

\end{proposition}

\begin{proof}
    The proof follows from the discussion above the proposition and fixed point argument using \eqref{GFL8}. 
\end{proof}

\begin{remark}
    For applications below, to calculate the perturbed eigenvalues for $L$ around $\lambda=1$, our main task is to compute the map $I+A+B(I-\Lambda)^{-1}C: E\to E$ up to $O(\epsilon^3)$. Once we obtained the formula for this map, it is straightforward to compute the eigenvalues since $E$ is finite dimensional.
\end{remark}

\subsection{Symmetry considerations}
We review some elementary observations concerning the symmetries of our equations. These will be useful for reducing the eigenvalue multiplicitites. 

Let us write
\begin{equation}\label{CPE1}
    \N(u):=-\Delta u+u\cdot\nabla u+\nabla p=-\Delta u+\pr(u\cdot\nabla u),
\end{equation}
where $\pr$ is the Helmholtz projection. We assume that $u$ are divergence-free fields defined on $\R^3$. We consider the usual action of the orthogonal group $O(n)$ on the vector fields: for $Q\in O(n)$ we let $Q\cdot u$ be defined as
\begin{equation}\label{CPE2}
    (Q\cdot u)(x):=Qu(Q^{-1}x).
\end{equation}

We remark that the definition also works when $u$ is defined in a subset $\Omega\subset\R^3$ which is invariant under $Q$. 

The Navier-Stokes equation is invariant under this action:
\begin{equation}\label{CPE3}
    \N(Q\cdot u)=Q\cdot \N(u). 
\end{equation}

The symmetries which will be relevant for our purposes here are 
\begin{enumerate}[label=(\roman*)]
    \item The rotations about the $x_3$-axis
    \begin{equation}\label{CPE4}
        R=R(\phi)=\begin{pmatrix}
            \cos\phi&-\sin\phi&0\\
            \sin\phi&\cos\phi&0\\
            0&0&1
        \end{pmatrix};
    \end{equation}

    \item The reflection
   \begin{equation}\label{CPE5}
   T=\begin{pmatrix}
        1&0&0\\
        0&1&0\\
        0&0&-1
    \end{pmatrix};
    \end{equation}

    \item  The reflection
   \begin{equation}\label{CPE6}
  S= \begin{pmatrix}
        -1&0&0\\
        0&1&0\\
        0&0&1
    \end{pmatrix}.
    \end{equation}
\end{enumerate}

    We will say that a vector field $u$ is 
\begin{enumerate}[label=(\roman*)]
    \item {\it axi-symmetric},
    when $R\cdot u=u$ for each $R=R(\phi)$ as in \eqref{CPE4};

    \item {\it no swirl},
    when it is axi-symmetric and, moreover, $S\cdot u=u$;

    \item {\it pure swirl},
    when it is axi-symmetric and $S\cdot u=-u$. 
\end{enumerate}


We will consider the Landau solutions in a form similar to that presented in Landau-Lifshitz \cite{Lifschitz}, pages 81 - 85. See also \eqref{Int3.2} - \eqref{Int3.3}. In polar coordinates, recall that 
\begin{equation}\label{CPE7}
    U^b(x):=\frac{1}{|x|}V^b(x/|x|)e_\theta+\frac{1}{|x|}F^b(x/|x|)e_r,
\end{equation}
and
\begin{equation}\label{CPE8}
\begin{split}
  &  V^b=-\frac{2\epsilon\sin\theta}{1-\epsilon\cos\theta},\quad F^b=2\bigg[\frac{1-\epsilon^2}{(1-\epsilon\cos\theta)^2}-1\bigg], \\
  &|b|=f(\epsilon):=16\pi\bigg[\frac{1}{\epsilon}+\frac{1}{2\epsilon^2}
  \log\frac{1-\epsilon}{1+\epsilon}+\frac{4\epsilon}{3(1-\epsilon^2)}\bigg],\quad \epsilon\in(0,1).  
    \end{split}
\end{equation}
The parameter $\epsilon$ is in $(-1,1)$, but we will only deal with small values of $\epsilon$ most of the time. To emphasize the dependence on $\epsilon$, we will write (recall the definition of the linear operator $L$ from \eqref{S1e9})
\begin{equation}\label{CPE9}
    U^\epsilon=U^b,\,\, F^\epsilon=F^b,\,\, V^\epsilon=V^b,\,\,L^\epsilon=L. 
\end{equation}
Clearly the Landau solutions are axi-symmetric with no swirl. Moreover, we have 
\begin{equation}\label{CPE10}
    U^{-\epsilon}=T\cdot U^{\epsilon}.
\end{equation}

The linearized operator $L^\epsilon$ in some sense inherits the symmetries of the corresponding nonlinear operators and the solutions at which we linearize. In particular, taking the derivative of 
\begin{equation}\label{CPE11}
    \N(T\cdot u)=T\cdot \N(u)
\end{equation}
and using $U^{-\epsilon}=T\cdot U^\epsilon$, it is not hard to see that 
\begin{equation}\label{CPE12}
    L^{-\epsilon}=TL^{\epsilon}T^{-1}=TL^{\epsilon}T. 
\end{equation}
This means that the spectra of $L^\epsilon$ and $L^{-\epsilon}$ are the same. 

\begin{remark}
    From the $T-$symmetry and the formula $L^{-\epsilon}=TL^\epsilon T$ it is clear that we need to compute various terms up to the second order (at least). We will not see a change of the eigenvalues at the first order. 
\end{remark}

In a similar way, using 
\begin{equation}\label{CPE13}
    \N(S\cdot u)=S\cdot \N(u),\quad S\cdot U=U \,\,({\rm where\,\,} U\,\,{\rm is\,\,a\,\,Landau\,\,solution})
\end{equation}
we obtain 
\begin{equation}\label{CPE14}
    L\cdot S=S\cdot L. 
\end{equation}
This implies

\begin{lemma}\label{CEP15}
    The pure swirl fields are invariant under $L$.
\end{lemma}

\begin{proof}
    The pure swirl fields are characterized by $S\cdot v=-v$, then $S\cdot L\cdot v=L\cdot S\cdot v=-L\cdot v$. 
\end{proof}

\subsection{Decomposing into angular Fourier modes} 
When ${\rm dim}\,E=1$, then $I+A+B(I-\Lambda)^{-1}C$ from Proposition \ref{GFL13.1} is just a number, exactly the sought-after eigenvalue (up to error $O(\epsilon^3)$). 

For our purposes, however, the eigenspace for $L_0$ corresponding to $\lambda=1$ has dimension $8$, see Proposition \ref{S2P1}. Therefore, we would need to do the calculation for ${\rm dim}\,E=8$, and to deal with $8\times8$ matrices which would be quite laborious. Fortunately, the dimension of the space $E$ can be reduced, thanks to the axial-symmetry of the Landau solutions, by restricting the spaces to different Fourier modes in the angular variable $\phi$.  

Recall the spherical coordinate $r, \theta,\phi$ from \eqref{Int3.1}-\eqref{Int3.10}, and the rotation operator $R=R(\alpha)$ from \eqref{CPE4}. We note that the fields $e_r, e_\theta, e_\phi$ are invariant under the action $u\to R\cdot u$ (with $(R\cdot u)(x)=Ru(R^{-1}x)$, and from this we see that 
\begin{equation}\label{GFL30}
\begin{split}
&R(\phi_0)\big[u^r(r,\theta,\phi)e_r+u^\theta(r,\theta,\phi)+u^\alpha(r,\theta,\alpha)e_\alpha\big]\\
&=u^r(r,\theta,\phi-\phi_0)e_r+u^\theta(r,\theta,\phi-\phi_0)+u^\alpha(r,\theta,\phi-\phi_0)e_\alpha.
\end{split}
\end{equation}
We will complexify the space $\mathcal{D}$ and $X$, and decompose according to Fourier modes in $\phi$. For $m\in\mathbb{Z}$, we call a vector field $v$, or scalar field $\varphi$ on $\us$, $m-$equivariant if for all $\phi,\phi_0\in[0,2\pi]$,
\begin{equation}\label{GFL31}
    R(\phi_0)\cdot v=v e^{-im\phi_0}, \quad \varphi(\phi-\phi_0)=\varphi(\phi)e^{-im\phi_0}.
\end{equation}

For $u$ as in \eqref{GFL30}, this amounts to the coordinate functions $u^r,u^\theta,u^\phi$ satisfying
\begin{equation}\label{GFL32}
    u^r(r,\theta,\phi)=u^r(r,\theta) e^{im\phi},\quad u^\theta(r,\theta,\phi)=u^\theta(r,\theta)e^{im\phi},\quad u^\phi(r,\theta,\phi)=u^\phi(r,\theta)e^{im\phi}.
\end{equation}

\begin{definition}\label{GFL32.1}
We define for $m\in\mathbb{Z}$ the subspaces $X^m$ and $\mathcal{D}^m$ of $X$ and $\mathcal{D}$ respectively, by requiring that all components of $\Xi\in X$ be $m-$equivariant.

The subspaces $X^m$ and $\mathcal{D}^m$ are invariant under our operators. Recall from \eqref{S2e15} that $E_1$ is the eigenspace of $L_0$ corresponding to the eigenvalue $\lambda=1$ (see also \ref{remark4.1} for the more explicit form). We define $E_1^{(m)}:=E_1\cap X^m$, and $E_{\rm axi}:=E_1^{(0)}.$
\end{definition}

By the invariance of $X^m$ and $\mathcal{D}^m$ we can apply Proposition \ref{GFL13.1} in these spaces. In practical calculations however, the distinction between $X$ and $X^m$ is minimal, since by the axial-symmetry of the Landau solutions and the corresponding linearized operators, the equivariance is preserved.

The vector field $\xi$ on $\us$ can be represented as 
\begin{equation}\label{GFL33}
    \xi=\xi^\theta e_\theta+\xi^\phi e_\phi,
\end{equation}
and we have the following formulae (where $\psi:\us\to\mathbb{C}$ is a scalar function)
\begin{equation}\label{GFL34}
\begin{split}
     &\nabla\psi=\frac{\partial\psi}{\partial\theta}e_\theta+\frac{\partial\psi}{\st\partial\phi}e_\phi,\quad \nabla^\perp \psi=-\frac{\partial\psi}{\st\partial\phi}e_\theta+\frac{\partial\psi}{\partial\theta}e_\phi,\\
    &{\rm div}\,\xi=\frac{1}{\st}\frac{\partial}{\partial\theta}\big(\st\, \xi^\theta\big)+\frac{\partial\xi^\phi}{\st\, \partial\phi},\quad {\rm curl}\,\xi=\frac{1}{\st}\frac{\partial}{\partial\theta}\big(\st\, \xi^\phi\big)-\frac{\partial\xi^\theta}{\st \partial \phi},\\
    &{\rm curl}\,\nabla^\perp\psi=\Delta\psi=\frac{1}{\st}\frac{\partial}{\partial\theta}\big(\st \frac{\partial\psi}{\partial\theta}\big)+\frac{\partial^2\psi}{\sin^2\theta\partial^2\phi}. 
    \end{split}
\end{equation}

We note that the vector fields $e_\theta, e_\phi$ are not smooth at $\theta=0$ and $\theta=\pi$, and hence even for smooth vector field $\xi$, the coordinate functions $\xi^\theta,\xi^\phi$ may not be smooth. 

One can also work with another representation
\begin{equation}\label{GFL35}
    \xi=\nabla \varphi+\nabla^\perp \psi, \quad \Delta\varphi={\rm div}\,\xi, \quad \Delta\psi={\rm curl}\,\xi. 
\end{equation}
In this representation, smooth vector fields $\xi$ are represented by smooth functions $\varphi, \psi$ and vice-versa. 

Using \eqref{GFL35}, we can expand expand the representation \eqref{S1e11}-\eqref{S1e11.1} for $\Xi\in X$ or $\mathcal{D}$ into a vector
\begin{equation}\label{GFL36}
    \Xi=\begin{pmatrix}
        \xi\\ \xi'\\\vartheta\\\vartheta^\ast
    \end{pmatrix}=\begin{pmatrix}
\varphi\\\psi\\\varphi'\\\psi'\\\vartheta\\\vartheta^\ast
    \end{pmatrix},\quad{\rm with}\,\,\xi\sim\begin{pmatrix}
        \varphi\\\psi
    \end{pmatrix},\,\,\xi'=\begin{pmatrix}
        \varphi'\\\psi'
    \end{pmatrix}. 
\end{equation}
In this representation, all components of the vector $\Xi$ are now scalar functions (compare also with the representation \eqref{ProP3}, which is equivalent using \eqref{GFL35}). 

In the following sections, we will calculate the perturbed eigenvalues of $L$ around $\lambda=1$ in each space $X^m$, $m\in\{0,1,2\}$ (the cases $m\in\{-1,-2\}$ then follow by conjugation). We note that ${\rm dim}\,E_{\rm axi}=2 $, ${\rm dim}\,E_1^{(1)}={\rm dim}\,E_1^{(-1)}=2$ and ${\rm dim}\,E_1^{(2)}={\rm dim}\,E_1^{(-2)}=1$. Our calculations below show that $L|_{E_{\rm axi}}$ has a geometrically double eigenvalue $\lambda=1$ (independent of $\epsilon$), $L|_{E_{1}^{(1)}}$ has eigenvalues $\lambda=1$ and $\lambda=1+\frac{1}{15}\epsilon^2+O(\epsilon^3)$, and $L|_{E_1^{(2)}}$ has eigenvalues $\lambda=1+\frac{4}{15}\epsilon^2+O(\epsilon^3)$. Moreover, the eigenvectors in each $E_1^{(m)}$ form nontrivial angles (independent of $\epsilon$). The Proposition \ref{mprop1} will then be proved. 

\section{Computation of the perturbed eigenvalues (II): The axially symmetric case $m=0$}\label{sec:m0}
In this section we compute the eigenvalues of $L$ which are close to $\lambda=1$, when restricted to axi-symmetric fields. From Remark \ref{rem1}, we see that $L$ has one axi-symmetric eigenvector corresponding to $\lambda=1$, which is generated by $\partial_{x_3}U^\epsilon$ through the formulae \eqref{S1e4} and \eqref{S1e7}. Since the space $E_{\rm axi}$ is two dimensional, to completely determine the spectrum of $L$ restricted to this space, we need to find the second eigenvector, which turns out to be given by the ``pure-swirl" fields. 

\subsection{Representation of pure swirl fields on the sphere}

Let $\theta, \phi$ be the spherical coordinates from \eqref{Int3.1}. Recall the definition of $e_\theta, e_\phi$ from \eqref{Int3.10} (corresponding to $\partial_\theta$ and $\frac{1}{\sin\theta }\partial_\phi$ respectively).  Clearly the pure swirl fields $\xi$ on the sphere can be represented as 
\begin{equation}\label{CEP16}
    \xi=v(\theta)e_\phi,
\end{equation}
where $v(\theta)$ is an ``axi-symmetric function" on the sphere. The advantage of this representation is that we are dealing with a scalar function $v$. A disadvantage is that smooth fields $\xi$ do not correspond to a smooth function $v(\theta)$, due to the singularity of the vector field $e_\phi$ at $\theta=0$. 

Noting that the fields $\xi=v(\theta)e_\phi$ are divergence free, we can represent $\xi$ by stream functions:
\begin{equation}\label{CEP17}
    \xi=\nabla^\perp \psi=\psi'(\theta)e_\phi \,\,({\rm where}\,\,\psi'=\frac{\partial\psi}{\partial\theta}). 
\end{equation}
Here $v(\theta)=\psi'(\theta)$. When $\xi$ is smooth, so is $\psi$ (as a function on $\us$). 

It is not hard to check that 
\begin{equation}\label{CEP18}
    \Delta \nabla^\perp\psi=\nabla^\perp\Delta\psi 
\end{equation}
and therefore the decomposition of the pure swirl fields into the eigenfunction of the (Hodge) Laplacian is also reflected by the decomposition of the axi-symmetric stream functions into the eigenfunctions of the Laplacian (on functions). These eigenfunctions are given by $P_k(\cos\theta), k=0,1,2,\dots,$ where $P_k$ are the Legendre polynomials, see \eqref{S31e4} for the precise formulae. 

For our calculations below we will use the formulae for $P_1, P_2, P_3$:
\begin{equation}\label{CEP19}
    P_1(t)=t,\quad P_2(t)=\frac{1}{2}(3t^2-1),\quad P_3(t)=\frac{1}{2}(5t^3-3t). 
\end{equation}


From the symmetry considerations, we see that the subspace of $\mathcal{D}$ given by 
\begin{equation}\label{CEP20}
    W:=\{(\xi,\xi',0,0):\,\xi,\,\xi'\,\,{\rm are\,\,pure\,\,swirl\,\,fields}\}
\end{equation}
is invariant under $L$ (this can also be checked by direct computation) and for our calculation we can restrict our attention to this subspace. For any
\begin{equation}\label{CEP21}
    \Xi=\begin{pmatrix}
        \xi\\\xi'\\0\\0
    \end{pmatrix}\in W,
\end{equation}
we represent the fields $\xi,\xi'$ by their stream functions, $\xi=\nabla^\perp \psi, \xi'=\nabla^\perp\psi'$, and write for simplicity of notations
\begin{equation}\label{CEP22}
    \Xi=\begin{bmatrix}
        \psi\\
        \psi'
    \end{bmatrix}.
\end{equation}
The space $W$ has a Hilbert basis consisting of the eigenvectors of $L_0$. In the stream function representation (see \eqref{CEP22}), it is
\begin{equation}\label{CEP23}
    \begin{split}
        &\Xi_k=\begin{bmatrix}
            p_k\\
            -kp_k
        \end{bmatrix}, \quad p_k(\theta)=P_k(\cos\theta),\,\,k=1,2,3,\dots, L_0\Xi_k=k\Xi_k;\\
        &\Xi_{-k}=\begin{bmatrix}
            p_k\\
            (k+1)p_k
        \end{bmatrix}, \quad p_k(\theta)=P_k(\cos\theta),\,\,k=1,2,3,\dots, L_0\Xi_{-k}=-(k+1)\Xi_{-k}.
    \end{split}
\end{equation}

\subsubsection{Calculations on the pure swirl fields}

We now turn to the calculation. All the objects will be calculated up to an error of order $O(\epsilon^3)$, where $\epsilon$ is the parameter in the Landau solution (see \eqref{CPE8}). 

For the Landau solution we have
\begin{equation}\label{GFL14}
    V^\epsilon=-2\epsilon\sin\theta-2\epsilon^2\cos\theta\,\sin\theta+O(\epsilon^3),\,\, F^\epsilon=4\epsilon \cos\theta+6\epsilon^2\cos^2\theta-2\epsilon^2+O(\epsilon^3).
\end{equation}

In our specific situation, the subspace $E$ from the general considerations above is the one-dimensional subspace generated by the vector 
$\Xi=\begin{bmatrix}
    p_1\\-p_1
\end{bmatrix}$. We will be using the following formulae:
\begin{enumerate}[label=(\roman*)]
    \item For $\xi=a(\theta)e_\phi$, we have $e_\theta \nabla \xi=\partial_\theta a\,e_\theta$;

    \item For $\xi=a(\theta)e_\phi$ and $H=H(\theta)$, we have $\xi\nabla (He_\theta)=H\cot\theta \,e_\phi$;

    \item $K\begin{pmatrix}
        \xi\\\xi'\\0\\0
    \end{pmatrix}=\begin{pmatrix}
        0\\-\bV^\epsilon  \nabla \xi-\xi\nabla \bV^\epsilon-F^\epsilon\xi'\\0\\0\end{pmatrix}$;

    \item $\Xi_1=\begin{bmatrix}
        p_1\\-p_1
    \end{bmatrix}=\begin{pmatrix}
        -\sin\theta \,e_\phi\\ \sin\theta \,e_\phi\\0\\0
    \end{pmatrix}$. 
\end{enumerate}

We calculate 
\begin{equation}\label{GFL15}
    K\Xi_1=\begin{pmatrix}
        0\\ \Big[-8\epsilon \cos\theta \,\sin\theta-10\epsilon^2\cos^2\theta\,\sin\theta+2\epsilon^2\sin\theta\Big]e_\phi+O(\epsilon^3)\\0\\0\end{pmatrix}.
\end{equation}
We express the term in the stream function representation using the following identities
\begin{equation}\label{GFL16}
\begin{split}
    &\frac{\partial}{\partial \theta} P_1(\cos\theta)=-\sin\theta,\quad \frac{\partial}{\partial \theta} P_2(\cos\theta)=-3\cos\theta\,\sin\theta,\\
    &\frac{\partial}{\partial \theta}P_3(\cos\theta)=-\frac{15}{2}\cos^2\theta\,\sin\theta+\frac{3}{2}\sin\theta,\\
    & \frac{\partial}{\partial \theta}\Big[\frac{8}{3}\epsilon P_2(\cos\theta)+\frac{4}{3}\epsilon^2P_3(\cos\theta)\Big]=-8\epsilon\cos\theta\,\sin\theta-10\epsilon^2\cos^2\theta\,\sin\theta+2\epsilon^2\sin\theta,
\end{split}
    \end{equation}

\begin{equation}\label{GFL17}
\begin{split}
    &\begin{bmatrix}
        0\\ p_1
    \end{bmatrix}=\frac{1}{3}\begin{bmatrix}
        p_1\\2p_1
    \end{bmatrix}-\frac{1}{3}\begin{bmatrix}
        p_1\\-p_1
    \end{bmatrix}=\frac{1}{3}\Xi_{-1}-\frac{1}{3}\Xi_1,\\
    &\begin{bmatrix}
        0\\ p_2
    \end{bmatrix}=\frac{1}{5}\begin{bmatrix}
        p_2\\3p_2
    \end{bmatrix}-\frac{1}{5}\begin{bmatrix}
        p_2\\-2p_2
    \end{bmatrix}=\frac{1}{5}\Xi_{-2}-\frac{1}{5}\Xi_2,\\
    &\begin{bmatrix}
        0\\p_3
    \end{bmatrix}=\frac{1}{7}\begin{bmatrix}
        p_3\\4p_3
    \end{bmatrix}-\frac{1}{7}\begin{bmatrix}
        p_3\\-3p_3
    \end{bmatrix}=\frac{1}{7}\Xi_{-3}-\frac{1}{7}\Xi_3.
\end{split}
\end{equation}
Putting these together, we see that
\begin{equation}\label{GFL18}
    K\Xi_1=\frac{8}{15}\epsilon\big(\Xi_{-2}-\Xi_2\big)+\frac{4}{21}\epsilon^2\big(\Xi_{-3}-\Xi_3\big)+O(\epsilon^3).
\end{equation}
We conclude that 
\begin{equation}\label{GFL19}
    A=O(\epsilon^3). 
\end{equation}

It remains to compute $B(I-\Lambda)^{-1}C$. 

As $B$ is of order $\epsilon$ and we calculate up to $O(\epsilon^3)$, we see from the formula for $K\Xi_1$ that we only need to calculate the $\Xi_1-$ component of 
\begin{equation}\label{GFL20}
    \frac{8}{15}\epsilon K(I-\Lambda)^{-1}\big(\Xi_{-2}-\Xi_2\big). 
\end{equation}
Recalling that $\Lambda \Xi_2=-3\Xi_{-2}$, $\Lambda \Xi_2=2\Xi_2$, we need to compute $\frac{8}{15}\epsilon K\big(\frac{1}{4}\Xi_{-2}+\Xi_2\big)$.

We can write 
\begin{equation}\label{GFL21}
\frac{1}{4}\Xi_{-2}+\Xi_2=\frac{1}{4}\begin{bmatrix}
        p_2\\3p_2
    \end{bmatrix}-\frac{1}{5}\begin{bmatrix}
        p_2\\-2p_2
    \end{bmatrix}=\frac{5}{4}\begin{bmatrix}
        p_2\\-p_2
    \end{bmatrix}. 
\end{equation}
We need to compute $K\begin{bmatrix}
        p_2\\-p_2
    \end{bmatrix}$ up to order $O(\epsilon^2)$, which amounts to computing $\begin{pmatrix}
        0\\-{\bf V}^\epsilon\nabla\xi-\xi\nabla {\bf V}^\epsilon-F\xi'\\0\\0\end{pmatrix}$ with $\xi=\big(\frac{\partial}{\partial\theta}P_2(\cos\theta)\big)e_\phi$, $\xi'=-\xi$ using the approximation
\begin{equation}\label{GFL22}
    V^\epsilon=-2\epsilon\sin\theta+O(\epsilon^2),\quad F^\epsilon=4\epsilon\cos\theta+O(\epsilon^2). 
\end{equation}

Guessing that the projection of $K\begin{bmatrix}
        p_2\\-p_2
    \end{bmatrix}$ to $\R \Xi_1$ might vanish, we will ignore the various multiplicative factors, and do the calculation for $\xi=-\cos\theta\,\sin\theta \,e_\phi, \xi'=\cos\theta\,\sin\theta \,e_\phi$. 

We compute:
\begin{equation}\label{GFL23}
    \begin{split}
        &-{\bf V}^\epsilon \nabla\xi=V^\epsilon \partial_\theta \big(\ct\,\st\big)e_\phi=-2\epsilon\,\st\big(-\sin^2\theta+\cos^2\theta\big)e_\phi,\\
        &-\xi\nabla \bV^\epsilon=\ct\,\st\,V^\epsilon \cot\theta\, e_\phi=-2\epsilon \ct\,\sin^2\theta\cot\theta \,e_\phi=-2\epsilon\cos^2\theta\,\sin\theta\, e_\phi,\\
        &-F^\epsilon\xi'=-F^\epsilon\ct\,\st \,e_\phi=-4\epsilon\cos^2\theta \st\, e_\phi,\\
        &-{\bf V}^\epsilon\nabla\xi-\xi\nabla {\bf V}^\epsilon-F^\epsilon\xi'=-10\epsilon\cos^2\theta\st+2\epsilon\st\\
        &\qquad\qquad\qquad\qquad\qquad\,\,=2\epsilon\frac{\partial}{\partial \theta}\big(\frac{5}{3}\cos^3\theta-\ct\big)=\frac{4}{3}\epsilon\frac{\partial}{\partial \theta}P_3(\ct).
    \end{split}
\end{equation}
We see that $K\begin{bmatrix}
    p_2\\-p_2
\end{bmatrix}$ is a multiple of $\begin{bmatrix}
    0\\p_3
\end{bmatrix}$, and hence a linear combination of $\Xi_3$ and $\Xi_{-3}$. In particular, its projection to the one dimensional space $\R\cdot\Xi_1$ vanishes. We conclude that
\begin{equation}\label{GFL24}
    B(I-\Lambda)^{-1}C=0+O(\epsilon^2),
\end{equation}
and therefore, combing this with the previous calculation for $A$, we have
\begin{equation}\label{GFL25}
    I+A+B(I-\Lambda)^{-1}C=I+O(\epsilon^3).
\end{equation}

\begin{remark}
    Since the spectrum of $L^\epsilon$ and $L^{-\epsilon}$ is the same, by the symmetry considerations above, the third order term should vanish, and therefore the eigenvalue $1$ should be precise up to $O(\epsilon^4)$.

    A higher order calculation could certainly be done, but fortunately one can determine by a non-perturbative computation that the eigenvalue $\lambda=1$ indeed does not move. 
\end{remark}

\subsection{Non-perturbative calculation for pure-swirl fields}\label{sec:pureswirl}
Consider the linearized Navier Stokes around the Landau solution, see \eqref{S1e3}, and take 
\begin{equation}\label{GFL26}
    v=\frac{1}{r}g(\ct)\st\,e_\phi,\quad f\equiv0.
\end{equation}
Letting $\ct=t$, we see that $g$ satisfies
\begin{equation}\label{GFL27}
    -(1-t^2)g''(t)+4tg'(t)+\frac{2\epsilon(1-t^2)}{1-\epsilon t}g'(t)+\bigg[\frac{-4\epsilon t}{1-\epsilon t}-\frac{2(1-\epsilon^2)}{(1-\epsilon t)^2}+2\bigg]g(t)=0. 
\end{equation}
We wish to find a solution which is regular at $t=\pm1$. The singular points $t=1,-1,\frac{1}{\epsilon}$ of the equation are all ``regular singular pints". When we try to determine the usual Frobenius series at $x=\frac{1}{\epsilon}$, the ``indicial equation" comes out as
\begin{equation}\label{GFL28}
    t^2+t-2=0
\end{equation}
which has roots $t=1$ and $t=-2$. 

We try to seek solutions of the form $\frac{{\rm polynomial}(t)}{(1-\epsilon t)^2}$. Luck is on our side here and it turns out that 
\begin{equation}\label{GFL29}
    g(t)=\frac{c}{(1-\epsilon t)^2}, \quad c\in\mathbb{C}
\end{equation}
is a solution. We conclude that the eigenvalue $\lambda=1$ does not move. 

\section{Computation of the perturbed eigenvalues (III): The cases $m=1$ and $m=2$}\label{sec:comp3}

Let us now turn to vector fields that are not necessarily axi-symmetric. 
In the previous section we focused our attention on the axi-symmetric fields. Using definition \eqref{GFL32.1}, that corresponds to working with $X^m$ for $m=0$. We will now proceed with the case $m=1$. 

\subsection{Calculation for $m=1$}
Here we assume that 
\begin{equation}\label{GFL37}
    u^r(r,\theta,\phi)=u^r(r,\theta)e^{i\phi},\quad u^\theta(r,\theta,\phi)=u^\theta(r,\theta)e^{i\phi},\quad u^\phi(r,\theta,\phi)=u^\phi(r,\theta)e^{i\phi}.
\end{equation}
We will use the representation \eqref{GFL35}-\eqref{GFL36} above, and hence smooth vector field $\xi$ will be represented by smooth functions $\varphi,\psi$. In the mode with $m=1$, the functions $\varphi,\psi$ will depend on $\phi$ through $e^{i\phi}$.

A convenient basis (more precisely, Hilbert basis in a suitable Hilbert space) of the functions with $m=1$ is given by the spherical harmonics
\begin{equation}\label{GFL38}
    Z_1^1=\st\,e^{i\phi}, \quad Z_2^1=\ct\,\st\,e^{i\phi},\quad Z_3^1=(5\cos^2\theta-1)\st\,e^{i\phi},\dots
\end{equation}
We adopted the usual $L^2-$ normalization factors and in what follows we will write
\begin{equation}\label{GFL39}
    Z_1:=Z_1^1,\quad Z_2=Z_2^1, \quad Z_3:=Z_3^1
\end{equation}
as the functions $Z_m^\ell$ with $\ell\neq1$ will not be involved in the calculations in this section. 

Recall the definition of $E_k, k\in\mathbb{Z}$ from \eqref{S2e15} which are the eigenspace of $L_0$ corresponding to the eigenvalue $\lambda=k$. Set in this section
\begin{equation}\label{GFL40}
    E=E_1\cap X^{m}, \quad{\rm with}\,\,m=1. 
\end{equation}
Then ${\rm dim}\,E=2$. In the representation \eqref{GFL36} two convenient vectors of the basis of $E$ are given by
\begin{equation}\label{GFL41}
\Xi_1:=\begin{pmatrix}
    0\\Z_1\\0\\-Z_1\\0\\0
\end{pmatrix}
\quad {\rm and}\quad\Xi_1':=\begin{pmatrix}
    0\\0\\0\\0\\Z_2\\-3Z_2
\end{pmatrix}.
\end{equation}

\subsubsection{Eigenvector generated by translation symmetry}
Using the notation $U^\epsilon:=U^b$ with $b=f(\epsilon) \begin{pmatrix}
    0\\0\\1
\end{pmatrix}$, we have
\begin{equation}\label{GFL42}
    \N(U^\epsilon)=0.
\end{equation}
Taking a derivative $\frac{\partial}{\partial x_i}$ ($i=1,2,3$), we get that
\begin{equation}\label{GFL43}
    \N'(U^\epsilon)\frac{\partial U^\epsilon}{\partial x_i}=0. 
\end{equation}
Representing  $\frac{\partial U^\epsilon}{\partial x_i}$ which is $-2-$ homogeneous in the formulation \eqref{Pe3}-\eqref{Pe4}, we obtain 
\begin{equation}\label{GFL44}
    \Xi_i=\begin{pmatrix}
        \xi\\\xi'\\\vartheta\\\vartheta^\ast
    \end{pmatrix}\in E_1.
   \end{equation}
The space $\Sigma={\rm span}\,\big\{\Xi_1,\Xi_2,\Xi_3\big\}$ has dimension $3$. In $\Sigma$, the element corresponding to the mode $m=1$ (i.e., with $e^{i\phi}$ dependence on $\phi$ of the spherical coordinate component) is obtained from 
\begin{equation}\label{GFL45}
    \Big(\frac{\partial}{\partial x_1}+i\frac{\partial}{\partial x_2}\Big)U^\epsilon.
\end{equation}

We have
\begin{equation}\label{GFL46}
    \frac{\partial}{\partial x_1}+i\frac{\partial}{\partial x_2}=e^{i\alpha}\Big(\st\,\frac{\partial}{\partial r}+\ct\,\frac{\partial}{r\partial \theta}+i\frac{\partial}{r\st\,\partial\phi}\Big).
\end{equation}
When we apply this operator to $U^\epsilon=\frac{1}{r}\big(V^\epsilon e_\theta+F^\epsilon e_r\big)$, we use
\begin{equation}\label{GFL47}
\begin{split}
    &\frac{\partial}{\partial r}e_r=0,\quad \frac{\partial}{\partial r}e_\theta=0,\quad \frac{\partial}{\partial\theta} e_r=e_\theta,\quad \frac{\partial}{\partial\theta}e_\theta=-e_r,\\
    &\frac{\partial}{\st\,\partial\phi}e_r=e_\phi,\quad \frac{\partial}{\st\,\partial\phi}e_\theta=\cot\theta\,e_\phi,
    \end{split}
\end{equation}
where we assume that we are on $\us$. The pressure in the Landau solution is
\begin{equation}\label{GFL48}
    p^\epsilon=\frac{4\epsilon}{r^2}\frac{\ct-\epsilon}{(1-\epsilon \ct)^2}.
\end{equation}

We calculate 
\begin{equation}\label{GFL49}
\begin{split}
    &\Big(\frac{\partial}{\partial x_1}+i\frac{\partial}{\partial x_2}\Big)U^\epsilon\\
    &=\frac{e^{i\phi}}{r^2}\Big(\big[-F\st-V\ct+\partial_\theta F\ct\big]e_r+\big[-V\st+\partial_\theta V\ct+F\ct\big]e_\theta\\
    &\quad\qquad\,+\big[iV\cot\theta+iF\big]e_\phi\Big).
    \end{split}
\end{equation}
Let us look at 
\begin{equation}\label{GFL50}
    \xi=\xi^\theta e_\theta+\xi^\alpha e_\phi=e^{i\phi}\big[-V\st+\partial_\theta V\ct+F\ct\big]e_\theta+e^{i\phi}\big[iV\cot\theta+iF\big]e_\phi.
\end{equation}
Recall that
\begin{equation}\label{GFL51}
    \partial_\theta V+V\cot\theta+F=0.
\end{equation}
This gives
\begin{equation}\label{GFL52}
    \xi^\theta=-e^{i\phi}\frac{V}{\st}=-\frac{\partial}{\st\partial\phi}\big[-ie^{i\phi}V\big].
\end{equation}
We also have
\begin{equation}\label{GFL53}
    \xi^\phi=e^{i\phi}\big(iV\cot\theta+iF\big)=-e^{i\phi}\partial_\theta V=\frac{\partial}{\partial\theta}\big[-ie^{i\phi}V\big].
\end{equation}
We see that 
\begin{equation}\label{GFL54}
    \xi=\nabla^\perp\psi,\quad{\rm with}\,\,\psi=-ie^{i\phi}V.
\end{equation}
For $\vartheta=e^{i\phi}\big[-F\st-V\ct+\partial_\theta F\ct\big]$ we can write (using \eqref{GFL51}) 
\begin{equation}\label{GFL55}
    \vartheta=e^{i\phi}\big[\partial_\theta V\st+\partial_\theta F \ct\big].
\end{equation}

We have (recalling the formulae \eqref{GFL38}-\eqref{GFL39})
\begin{equation}\label{GFL56}
    \begin{split}
        &V=-2\epsilon\st-2\eps^2\ct\st+O(\eps^3),\quad F=4\eps\ct+6\eps^2\cos^2\theta-2\eps^2+O(\eps^3),\\
        &\partial_\theta V=-2\eps\ct-2\eps^2(2\cos^2\theta-1)+O(\eps^3),\\
        &\partial_\theta F=-4\eps\st-12\eps^2\ct\st+O(\eps^3),\\
        &\vartheta=-e^{i\phi}6\eps\ct\st+e^{i\phi}\eps^2\big[2\st-16\cos^2\theta\st\big]+O(\eps^3)\\
        &\,\,\,\,=-6\epsilon Z_2+\eps^2\Big[-\frac{16}{5}Z_3-\frac{6}{5}Z_1\Big]+O(\eps^3).
    \end{split}
\end{equation}
In this notation, we can also write
\begin{equation}\label{GFL57}
   \psi=-ie^{i\phi}V=2\eps iZ_1+2\eps^2iZ_2+O(\eps^3).  
\end{equation}

Let us now compute, using the formula \eqref{GFL48},
\begin{equation}\label{GFL58}
\begin{split}
   &\Big(\frac{\partial}{\partial x_1}+i\frac{\partial}{\partial x_2}\Big)p^\epsilon \\
   &=e^{i\alpha}\Big(\st\,\frac{\partial}{\partial r}+\ct\,\frac{\partial}{r\partial \theta}\Big)\Big[\frac{4\epsilon}{r^2}\frac{\ct-\epsilon}{(1-\epsilon \ct)^2}\Big]\\
   &=4\epsilon e^{i\phi}\Big(\st\,\frac{\partial}{\partial r}+\ct\,\frac{\partial}{r\partial \theta}\Big)\frac{\ct-\eps+2\eps\cos^2\theta}{r^2}+O(\eps^3)\\
   &=4\epsilon e^{i\phi}\frac{(-2\ct+2\epsilon-4\eps\cos^2\theta)\st-\ct\,\st-4\eps \cos^2\theta\,\st}{r^2}+O(\eps^2).
   \end{split}
\end{equation}
Therefore recalling \eqref{Pe3}, we have
\begin{equation}\label{GFL59}
\begin{split}
    q= \Big(\frac{\partial}{\partial x_1}+i\frac{\partial}{\partial x_2}\Big)p^\epsilon|_{r=1}&=4 e^{i\phi}\Big[-3\epsilon\ct\,\st-8\eps^2\cos^2\theta\,\st+2\eps^2\st\Big]\\
    &=\epsilon (-12Z_2)+\eps^2\Big(-\frac{32}{5}Z_3+\frac{8}{5}Z_1\Big)+O(\eps^3),
    \end{split}
\end{equation}
and (noting also that ${\rm div}\,\xi=0$ from \eqref{GFL54})
\begin{equation}\label{GFL60}
    \vartheta^\ast=-\vartheta-q=\epsilon \,18Z_2+\epsilon^2\Big(\frac{48}{5}Z_3-\frac{2}{5}Z_1\Big)+O(\eps^3).
\end{equation}

The following is then an eigenvector of $L$ with $\lambda=1$:
\begin{equation}\label{GFL61}
\begin{pmatrix}
    \xi\\\xi'\\\vartheta\\\vartheta^\ast
\end{pmatrix}=\begin{pmatrix}
    \varphi\\\psi\\\varphi'\\\psi'\\\vartheta\\\vartheta^\ast
\end{pmatrix}=\begin{pmatrix}
    0\\\epsilon\,2iZ_1+\epsilon^2\,2iZ_2\\0\\-2\epsilon iZ_1-\epsilon^2\,2iZ_2\\-\epsilon \,6Z_2+\eps^2\big[-\frac{16}{5}Z_3-\frac{6}{5}Z_1\big]\\\epsilon \,18Z_2+\eps^2\big[\frac{48}{5}Z_3-\frac{2}{5}Z_1\big]
\end{pmatrix}+O(\eps^3).
\end{equation}

\subsubsection{Calculation of $K\begin{pmatrix}
    0\\Z_1\\0\\-Z_1\\0\\0
\end{pmatrix}$}
We let $\Xi_1=\begin{pmatrix}
    0\\Z_1\\0\\-Z_1\\0\\0
\end{pmatrix}$, so $\xi=\nabla^\perp\psi=\nabla^\perp Z_1$. Simple calculations show
\begin{equation}\label{GFL62}
    \xi=-ie^{i\phi}e_\theta+\ct \,e^{i\phi}e_\phi,\quad \xi'=-\xi,
\end{equation}
Recalling the formula \eqref{S31e1} for the operator $K$, we have 
\begin{equation}\label{GFL63}
  K\Xi_1=\begin{pmatrix}
      0\\-{\bf V}\nabla\xi-\xi\nabla {\bf V}-F\xi'\\
      0\\-\xi\nabla F+2\xi\cdot{\bf V}
  \end{pmatrix}.  
\end{equation}
We use the identities with covariant derivative on $\us$,
\begin{equation}\label{GFL64}
    \nabla_{e_\theta} (Ve_\theta)=\frac{\partial V}{\partial \theta}\,e_\theta,\quad \nabla_{e_\phi}(V e_\theta)=V\cot\theta\,e_\phi.  
\end{equation}

{\it Step 1: The term ${\bf V}\nabla \xi$.}  We begin with the calculation of the term ${\bf V}\nabla \xi$. 
\begin{equation}\label{GFL65}
    {\bf V}\nabla\xi=V\nabla_{e_\theta}\big(-ie^{i\phi}e_\theta+\ct\,e^{i\phi}e_\phi\big)=-V\st \,e^{i\phi}e_\phi.
\end{equation}
We calculate the corresponding potential $\varphi_1$ and stream function $\psi_1$ as follows:
\begin{equation}\label{GFL66}
    \begin{split}
        &{\bf V}\nabla\xi=\nabla\varphi_1+\nabla^\perp\psi_1\Longrightarrow \Delta\varphi_1={\rm div}\,({\bf V}\nabla\xi),\quad \Delta\psi={\rm curl}\,({\bf V}\nabla\xi),\\
        &{\rm div}\,({\bf V}\nabla\xi)=\frac{1}{\st}\frac{\partial}{\partial\phi}\big(-V\st\,e^{i\phi}\big)=-iVe^{i\phi}\\
        &\qquad\qquad\,\,\,=-e^{i\phi}\big(-2\eps\,\st-2\eps^2\ct\,\st\big)+O(\eps^3),\\
        &{\rm div}\,({\bf V}\nabla\xi)=\epsilon\,2iZ_1+\eps^22iZ_2+O(\eps^3),\\
    \end{split}
\end{equation}
and
\begin{equation}\label{GFL67}
    \begin{split}
        &{\rm curl}\,({\bf V}\nabla\xi)=-e^{i\phi}\big[\partial_\theta V\st+2V\ct\big],\\
        &\partial_\theta V=\partial_\theta\big[-2\eps\,\st-2\eps^2\ct\,\st\big]=-2\eps\,\ct-4\eps^2\cos^2\theta+2\eps^2+O(\eps^3),\\
        &{\rm curl}\,({\bf V}\nabla\xi)=e^{i\phi}\big[\eps\,6\ct\,\st+\eps^28\cos^2\theta\,\st-\eps^22\st\big]+O(\eps^3)\\
        &\qquad\qquad\,\,\,\,=\eps \,6Z_3+\eps^2\Big[\frac{8}{5}Z_3-\frac{2}{5}Z_1\Big]+O(\eps^3). 
    \end{split}
\end{equation}
Using the identities
\begin{equation}\label{GFL68}
    \Delta Z_1=-2Z_1,\quad \Delta Z_2=-6Z_2,\quad \Delta Z_3=-12 Z_3. 
\end{equation}
In view of \eqref{GFL66}-\eqref{GFL68}, we see that
\begin{equation}\label{GFL69}
    {\bf V}\nabla\xi=\nabla\varphi_1+\nabla^\perp\psi_1, \quad \varphi_1=-\eps\,iZ_1-\eps^2\frac{1}{3}iZ_2,\quad \psi_1=-\eps\,Z_2+\eps^2\Big[-\frac{2}{15}Z_3+\frac{1}{5}Z_1\Big].
\end{equation}

{\it Step 2: The term $\xi\nabla {\bf V}$.} We next turn to the calculation of $\xi\nabla {\bf V}$. 
\begin{equation}\label{GFL70}
    \begin{split}
        \xi\nabla {\bf V}&=e^{i\phi}\big(-ie_\theta+\ct\,e_\phi\big)\nabla (Ve_\theta)=-ie^{i\phi}\partial_\theta V\,e_\theta+e^{i\phi}V\frac{\cos^2\theta}{\st}e_\phi,\\
        {\rm div}\,(\xi\nabla {\bf V})&=\frac{1}{\st}\frac{\partial}{\partial\theta}\Big[\st\,(-ie^{i\phi}\partial_\theta V)\Big]+\frac{1}{\st}\partial_\phi\Big(e^{i\phi}V\frac{\cos^2\theta}{\st}\Big)\\
        &=e^{i\phi}(-i)\Big[\epsilon\,2\st+O(\eps^3)\Big]=\eps\,(-2i)Z_1+\eps^2(-10i)Z_2+O(\eps^3).
    \end{split}
\end{equation}
Hence the potential $\varphi_2$ for $\xi\nabla {\bf V}$ is
\begin{equation}\label{GFL71}
    \varphi_2=\epsilon\,i\,Z_1+\eps^2i\,\frac{5}{3}Z_2+O(\eps^3). 
\end{equation}

We can also compute
\begin{equation}\label{GFL72}
    \begin{split}
        {\rm curl}\,(\xi\nabla {\bf V})&=\frac{1}{\st}\frac{\partial}{\partial\theta}\big(e^{i\phi}V\cos^2\theta\big)-\frac{1}{\st}\frac{\partial}{\partial\phi}\big(-ie^{i\phi}\partial_\theta V\big)\\
&=e^{i\phi}\big[6\epsilon\,\ct\,\st+\eps^2(8\cos^2\theta\,\st-2\st)\big]+O(\eps^3)\\
&=6\epsilon\,Z_2+\eps^2\Big(\frac{8}{5}Z_3-\frac{2}{5}Z_1\Big).
        \end{split}
\end{equation}
It follows that the strem function $\psi_2$ for $\xi\nabla {\bf V}$ is
\begin{equation}\label{GFL73}
    \psi_2=-\epsilon\,Z_2+\eps^2\big(-\frac{2}{15}Z_3+\frac{1}{5}Z_1\big).
\end{equation}

{\it Step 3: The term $F\xi'=-F\xi$.} We turn to the calculation of the term $F\xi'=-F\xi$. 
\begin{equation}\label{GFL74}
    \begin{split}
        F\xi'&=-F\xi=iFe^{i\phi}e_\theta-F\ct\,e^{i\phi}e_\phi,\\
        {\rm div}\,(F\xi')&=\frac{1}{\st}\frac{\partial}{\partial\theta}\big(\st\,ie^{i\phi} F\big)+\frac{\partial}{\st\,\partial}\big(-e^{i\phi}F\ct\big)\\
        &=ie^{i\phi}\Big[\partial_\theta F+F\cot\theta-F\cot\theta\Big]\\
        &=\epsilon(-4i)Z_1+\eps^2(-12i)Z_2+O(\eps^3).
    \end{split}
\end{equation}
As a consequence, the potential function $\varphi_3$ for $F\xi'$ is given by
\begin{equation}\label{GFL75}
\varphi_3=\epsilon(2i)Z_1+\eps^2(2i)Z_2+O(\eps^3).
\end{equation}

To get the stream function for the field $F\xi'$, using 
\begin{equation}
    F=4\epsilon\,\ct+6\eps^2\cos^2\theta-2\eps^2+O(\epsilon^3),\quad \partial_\theta F=-4\eps\,\st-12\eps^2\ct\,\st+O(\epsilon^3),
\end{equation}
we calculate
\begin{equation}\label{GFL76}
    \begin{split}
        {\rm curl}\,(F\xi')&=\frac{1}{\st}\frac{\partial}{\partial\theta}\Big(\st\,(-F\ct)e^{i\phi}\Big)-\frac{\partial}{\st\,\partial\phi}\big(iFe^{i\phi}\big)\\
        &=e^{i\phi}\Big[-\partial_\theta F\,\ct+2F\st\Big]\\
        &=e^{i\phi}\Big[12\eps\,\ct\st+24\eps^2\cos^2\theta\,\st-4\eps^2\st\Big]\\
        &=12\eps\,Z_2+\eps^2\frac{24}{5}Z_3+\eps^2\frac{4}{5}Z_1+O(\eps^3). 
    \end{split}
\end{equation}
Hence the stream function $\psi_3$ for $F\xi'$ is given by
\begin{equation}\label{GFL77}
    \psi_3=\epsilon(-2)Z_2+\eps^2\Big[-\frac{2}{5}Z_3-\frac{2}{5}Z_1\Big]. 
\end{equation}

{\it Step 4: Summary of calculation.} Summarizing the above calculations, we conclude that the stream function $\psi$ and potential function $\varphi$ for ${\bf V}\nabla\xi+\xi\nabla {\bf V}+F\xi'$ are given as follows:
\begin{equation}\label{GFL78}
    {\bf V}\nabla\xi+\xi\nabla {\bf V}+F\xi'=\nabla\varphi+\nabla^\perp\psi,\,\, \varphi=\epsilon (2i)Z_1+\epsilon^2\frac{10}{3} iZ_2,\,\, \psi=\epsilon (-4)Z_2+\eps^2\big(-\frac{2}{3}Z_3\big). 
\end{equation}

We now turn to the term
\begin{equation}\label{GFL79}
    \begin{split}
     \vartheta^\ast&=-\xi\nabla F+2\xi\cdot {\bf V} \\
&=-\xi^\theta\frac{\partial F}{\partial\theta}+2\xi^\theta V\\
&=\epsilon^2(-8i)Z_2.
    \end{split}
\end{equation}
In this calculation, we used the formulae
\begin{equation}\label{GFL80}
\begin{split}
   &\xi=-ie^{i\phi}e_\theta+\ct\,e^{i\phi}e_\phi,\,\,   F=4\epsilon\ct+6\eps^2\cos^2\theta-\eps^2+O(\eps^3),\\
   &\partial_\theta F=-4\epsilon\st-12\eps^2\ct\,\st+O(\eps^3),\,\, V=-2\eps\st-2\eps^2\ct\,\st+O(\eps^3).
\end{split}
\end{equation}

We conclude that
\begin{equation}\label{GFL81}
    K\Xi_1=\begin{pmatrix}
      0\\-{\bf V}\nabla\xi-\xi\nabla {\bf V}-F\xi'\\
      0\\-\xi\nabla F+2\xi\cdot{\bf V}
  \end{pmatrix}=\begin{pmatrix}
      0\\0\\\epsilon(-2i)Z_1+\eps^2\big(-\frac{10}{3}i\big)Z_2\\ \epsilon\,4Z_2+\eps^2\frac{2}{3}Z_3\\0\\\eps^2(-8i)Z_2
  \end{pmatrix}+O(\epsilon^3).
\end{equation}

\subsubsection{Calculating $A\Xi_1=P_1K\Xi_1$}
For calculations below we need to express $K\Xi_1$ in terms of the eigenvectors of $L_0$. Let us consider eigenvectors of $L_0$ generated by $Z_k$ (with $\Delta Z_k=-k(k+1)Z_k$), with vanishing stream functions. These will be of the form
$\begin{pmatrix}
    \gamma_1 Z_k\\0\\\gamma_2 Z_k\\0\\\gamma_3Z_k\\\gamma_4Z_k
\end{pmatrix}$ with eigenvalues from $\{k+1,k-1,-k,-k-2\}$, where for a given $k=1,2,3,\dots$ the coefficients $\gamma_1,\gamma_2,\gamma_3, \gamma_4$ are given by table \ref{tab1}.

\begin{table}
    \centering
    \begin{tabular}{|c|c|c|c|c|}
    \hline
         &$\lambda=k+1 $ &$\lambda= k-1 $& $\lambda=-k$ &$\lambda=-k-2$ \\
         \hline
        $\varphi$ & $-\frac{1}{k+1} $& $-\frac{k-2}{k(k+1)}$ & $\frac{1}{k}$ & $\frac{k+3}{k(k+1)}$ \\
      $ \varphi'  $& $1$ &$ \frac{(k-1)(k-2)}{k(k+1)} $ & $1$  &$ \frac{(k+2)(k+3)}{k(k+1)}$ \\
        $\vartheta$ &$1$  & $1$ & $1$ &$ 1$ \\
      $\vartheta^\ast$   &$ -(k+1)$ & $-\frac{k^2+4k-3}{k+1}$  & $k$ & $\frac{k^2-2k-6}{k}$ \\
       $q$  &$0$  &$\frac{4k-2}{k+1}$  &$0$  &$\frac{4k+6}{k}$ \\
       \hline
    \end{tabular}
    \caption{Eigenvectors generated by $Z_k$}
    \label{tab1}
\end{table}

For example, the vector $\begin{pmatrix}
    -\frac{1}{2}Z_1\\0\\Z_1\\0\\Z_1\\-2Z_1
\end{pmatrix}$ is an eigenvector of $L_0$ with eigenvalue $\lambda=k+1=2$ (and $q=0$). As another example, for $k=2$ and $\lambda=k-1$, we have the vector $\begin{pmatrix}
    0\\0\\0\\0\\Z_2\\-3Z_2
\end{pmatrix}$ with the corresponding $q=2Z_2$. Note that a scalar multiple of this vector came up as a part of the first-order approaximation of $(\partial_{x_1}+i\partial_{x_2})U^\epsilon$ in the calculation above. 

To determine $A\Xi_1=P_1(K\Xi_1)$, we first note from the tables \ref{tab1} and \ref{table3} that 
\begin{equation}\label{GFL84}
    P_1(K\Xi_1)=P_1\begin{pmatrix}
      0\\0\\\eps^2\big(-\frac{10}{3}i\big)Z_2\\ 0\\0\\\eps^2(-8i)Z_2
  \end{pmatrix},
\end{equation}
as the other components have no projection in the eigenspace $E_1$ of $L_0$. 

For $Z_2$, the table \ref{tab1} becomes
table \ref{table3}.

\begin{table}
    \centering
    \begin{tabular}{ccccc}
         &$\lambda=3$  &$\lambda=1$  &$\lambda=-2$  &$\lambda=-4$ \\
         \hline
       $\varphi$  &$-\frac{1}{3}$  & $0$ & $\frac{1}{2}$ &$\frac{5}{6}$ \\[0.5ex] 
       $\varphi'$  & $1$ & $0$ &$1$  &$\frac{10}{3}$ \\
       $\vartheta$  & $1$ & $1$ &$1$  &$1$ \\
      $\vartheta^\ast$   &$-3$  &$-3$  &$2$  &$-3$ \\
    \end{tabular}
    \caption{Eigenvectors generated by $Z_2$}
    \label{table3}
\end{table}
We compute
\begin{equation}\label{GFL85}
    \begin{pmatrix}
        0\\-\frac{10}{3}\\0\\-8
    \end{pmatrix}=-\frac{74}{35}\begin{pmatrix}
        -\frac{1}{3}\\1\\1\\-3
    \end{pmatrix}+\frac{18}{5}\begin{pmatrix}
        0\\0\\1\\-3
    \end{pmatrix}-\frac{8}{5}\begin{pmatrix}
        \frac{1}{2}\\1\\1\\2
    \end{pmatrix}+\frac{4}{35}\begin{pmatrix}
        \frac{5}{6}\\\frac{10}{3}\\1\\-3
    \end{pmatrix}.
\end{equation}
 Therefore 
 \begin{equation}\label{GFL86}
     P_1(K\Xi_1)=\eps^2\frac{18}{5}i\begin{pmatrix}
         0\\0\\0\\Z_2\\-3Z_2
     \end{pmatrix}. 
 \end{equation}

\subsubsection{Calculating $B(I-\Lambda)^{-1}C\,\Xi_1$}
 We now turn to the term $B(I-\Lambda)^{-1}C\,\Xi_1$. Since $B$ is of order $O(\epsilon)$, it suffices to compute $C\Xi_1$ up to order $O(\epsilon^2)$. From the definition, see \eqref{GFL3}-\eqref{GFL4}, we get that
 \begin{equation}\label{GFL87}
     C\,\Xi_1=(I-P_1)K\Xi_1=K\Xi_1+O(\epsilon^2).
 \end{equation}
 Hence the vector that we need to calculate is 
 \begin{equation}\label{GFL88}
  B(I-\Lambda)^{-1}C\,\Xi_1=P_1K(I-\Lambda)^{-1}K\Xi_1=P_1K(I-\Lambda)^{-1}\begin{pmatrix}
      0\\0\\\epsilon(-2i)Z_1\\\epsilon\,4Z_2\\0\\0
  \end{pmatrix},\quad {\rm mod}\,\,O(\eps^3).   \end{equation}

{\it Step 1: Computing $(I-\Lambda)^{-1}\begin{pmatrix}
0\\0\\\epsilon(-2i)Z_1\\\epsilon\,4Z_2\\0\\0
  \end{pmatrix}$.} 
  Let us start by computing $(I-\Lambda)^{-1}\begin{pmatrix}
    0\\0\\Z_1\\0\\0\\0
\end{pmatrix}$ using Table \ref{tab1} for the eigenvectors of $L_0$ generated by $Z_1$.
\begin{table}
    \centering
    \begin{tabular}{c|cccc}
         &$\lambda=2$  &$\lambda=0$  & $\lambda=-1$ & $\lambda=-3$\\
         \hline
        $\varphi$ &$-\frac{1}{2}$  &$\frac{1}{2}$  & $1$ &$2$ \\
      $\varphi'$   & $1$ &$0$  & $1$ &$6$ \\
     $\vartheta$    & $1$ & $1$ &$1$  &$1$ \\
     $\vartheta^\ast$    & $-2$ &$-1$  &$1$  &$-7$ \\
    \end{tabular}
    \caption{Eigenvectors generated by $Z_1$}
    \label{table4}
\end{table}

  We compute 
  \begin{equation}\label{GFL89}
      \begin{pmatrix}
          0\\1\\0\\0
      \end{pmatrix}=\frac{4}{15}\begin{pmatrix}
          -\frac{1}{2}\\1\\1\\-2
      \end{pmatrix}-\frac{2}{3}\begin{pmatrix}
          \frac{1}{2}\\0\\1\\-1
      \end{pmatrix}+\frac{1}{3}\begin{pmatrix}
          1\\1\\1\\1
      \end{pmatrix}+\frac{1}{15}\begin{pmatrix}
          2\\6\\1\\-7
      \end{pmatrix}.
  \end{equation}
  Thus
  \begin{equation}\label{GFL90}
  \begin{split}
      (I-\Lambda)^{-1}\begin{pmatrix}
          0\\0\\1\\0\\0\\0
      \end{pmatrix}Z_2=&-1\cdot\frac{4}{15}\begin{pmatrix}
          -\frac{1}{2}\\0\\1\\0\\1\\-2
      \end{pmatrix}Z_2-\frac{2}{3}\begin{pmatrix}
          \frac{1}{2}\\0\\0\\0\\1\\-1
      \end{pmatrix}Z_2+\frac{1}{2}\cdot\frac{1}{3}\begin{pmatrix}
          1\\0\\1\\0\\1\\1
      \end{pmatrix}Z_2\\
      &+\frac{1}{4}\cdot\frac{1}{15}\begin{pmatrix}
          2\\0\\6\\0\\1\\-7
      \end{pmatrix}Z_2=\begin{pmatrix}
          0\\0\\0\\0\\-\frac{3}{4}\\\frac{5}{4}
      \end{pmatrix}Z_2.
      \end{split}
  \end{equation}

We next compute $(I-\Lambda)^{-1}\begin{pmatrix}
      0\\0\\0\\Z_2\\0\\0
  \end{pmatrix}$. 
Using the identity
\begin{equation}\label{GFL91}
  \begin{pmatrix}
      0\\0\\0\\Z_2\\0\\0
  \end{pmatrix}=-\frac{1}{5}\begin{pmatrix}
      0\\Z_2\\0\\-2Z_2\\0\\0
  \end{pmatrix} +\frac{1}{5}\begin{pmatrix}
      0\\Z_2\\0\\3Z_2\\0\\0
  \end{pmatrix}, 
\end{equation}
where the first vector belongs to $E_2$ and the second belongs to $E_{-3}$, we see that
\begin{equation}\label{GFL92}
    (I-\Lambda)^{-1}\begin{pmatrix}
      0\\0\\0\\Z_2\\0\\0
  \end{pmatrix}=-1\cdot\big(-\frac{1}{5}\big)\begin{pmatrix}
      0\\Z_2\\0\\-2Z_2\\0\\0
  \end{pmatrix} +\frac{1}{4}\cdot\frac{1}{5}\begin{pmatrix}
      0\\Z_2\\0\\3Z_2\\0\\0
  \end{pmatrix}=\begin{pmatrix}
      0\\\frac{1}{4}Z_2\\0\\-\frac{1}{4}Z_2\\0\\0
  \end{pmatrix}.
\end{equation}

Summarizing the results \eqref{GFL90} and \eqref{GFL92}, we see that
\begin{equation}\label{GFL93}
    (I-\Lambda)^{-1}\begin{pmatrix}
      0\\0\\\epsilon(-2i)Z_1\\\epsilon\,4Z_2\\0\\0
  \end{pmatrix}=\begin{pmatrix}
       0\\\epsilon Z_2\\0\\-\epsilon Z_2\\\epsilon\frac{3i}{2}Z_1\\\epsilon\frac{-5i}{2}Z_1 
    \end{pmatrix}.
\end{equation}

{\it Step 2: computing $P_1K\begin{pmatrix}
       0\\\epsilon Z_2\\0\\-\epsilon Z_2\\\epsilon\frac{3i}{2}Z_1\\\epsilon\frac{-5i}{2}Z_1 
    \end{pmatrix}$. } 
We first calculate, using formula \eqref{S31e1} and \eqref{GFL80},
\begin{equation}\label{GFL94}
    K\begin{pmatrix}
        0\\0\\0\\0\\\frac{3i\epsilon}{2}Z_1\\-\frac{5i\eps}{2}Z_1
    \end{pmatrix}=\begin{pmatrix}
        0\\0\\0\\0\\0\\-V\vartheta+(2\vartheta-\vartheta^\ast-q)F 
    \end{pmatrix}=\begin{pmatrix}
        0\\0\\0\\0\\0\\\eps^221 iZ_2
    \end{pmatrix},
\end{equation}
where we used the computation
\begin{equation}\label{GFL95}
\begin{split}
   & -v\nabla\vartheta=-(-2\eps\,\st)\frac{\partial}{\partial\theta}\Big(\frac{3}{2}i\epsilon \st\,e^{i\phi}\Big)=\eps^23i\st\,\ct\,e^{i\phi}=\eps^23iZ_2,\\
&(2\vartheta-\vartheta^\ast-q)F=4\eps\,\ct\,\frac{9}{2}i\eps\,\st e^{i\phi}=\eps^218iZ_2. 
    \end{split}
\end{equation}

Using the identity
\begin{equation}\label{GFL96}
    \begin{pmatrix}
        0\\0\\0\\1
    \end{pmatrix}=\frac{3}{35}\begin{pmatrix}
        -\frac{1}{3}\\1\\1\\-3
    \end{pmatrix}-\frac{1}{5}\begin{pmatrix}
        0\\0\\1\\-3
    \end{pmatrix}+\frac{1}{5}\begin{pmatrix}
        \frac{1}{2}\\1\\1\\2
    \end{pmatrix}-\frac{3}{35}\begin{pmatrix}
        \frac{5}{6}\\\frac{10}{3}\\7\\-3
    \end{pmatrix},
\end{equation}
it follows from Table \ref{table3} and simple calculation that
\begin{equation}\label{GFL97}
    P_1K\begin{pmatrix}
        0\\0\\0\\0\\\frac{3i\epsilon}{2}Z_1\\-\frac{5i\eps}{2}Z_1
    \end{pmatrix}=\epsilon^2i\Big(-\frac{21}{5}\Big)\begin{pmatrix}
        0\\0\\0\\0\\Z_2\\-3Z_2
    \end{pmatrix}.
\end{equation}

Next we calculate $P_1K\begin{pmatrix}
    0\\Z_2\\0\\-Z_2\\0\\0
\end{pmatrix}$. Recall the formula \eqref{S31e1} for $K$ and \eqref{GFL80} for $V, F$. We have
\begin{equation}\label{GFL98}
    K\begin{pmatrix}
        \xi\\\xi'\\0\\0
    \end{pmatrix}=\begin{pmatrix}
        0\\-{\bf V}\nabla \xi-\xi\nabla {\bf V}-F\xi'\\0\\-\xi\nabla F+2\xi\cdot {\bf V}
    \end{pmatrix},\quad \xi=\nabla^\perp Z_2,\quad \xi'=-\nabla^\perp Z_2.
\end{equation}

We consider first the term ${\bf V}\nabla\xi$: 
\begin{equation}\label{GFL99}
    \begin{split}
\bV\nabla\xi&=V\nabla_{e_\theta}\Big(\xi^\theta\, e_\theta+\xi^\phi e_\phi\Big)=V\frac{\partial\xi^\theta}{\partial\theta}e_{\theta}+V\frac{\partial\xi^\phi}{\partial\theta}e_\phi\\
        &=(-2\eps\,\st)i\st \,e^{i\phi}e_\theta+(-2\epsilon\,\st)(-4\ct\,\st)e^{i\phi}e_\phi\\
        &=e^{i\phi}\Big[\eps (-2i)\sin^2\theta e_\phi+\epsilon\,8\ct\,\sin^2\theta \,e_\phi\Big].
    \end{split}
\end{equation}
As before, we wish to determine the stream and potential functions of the field ${\bf V}\nabla \xi$, and hence we calculate 
\begin{equation}\label{GFL100}
    \begin{split}
        {\rm div}\,({\bf V}\nabla\xi)&=e^{i\phi}\frac{-2i\eps}{\st}\frac{\partial}{\partial\theta}\Big(\st\,\sin^2\theta\Big)+\frac{8\eps\partial}{\st\partial\phi}\Big(e^{i\phi}\ct\,\sin^2\theta\Big)\\
        &=e^{i\phi}(-6i\eps)\st\,\ct+e^{i\phi}(8i\eps)\st\,\ct=2i\eps Z_2,\\
 {\rm curl}\,({\bf V}\nabla\xi)&=\frac{e^{i\phi}}{\st}\frac{\partial}{\partial\theta}\Big(\st\,8\eps\,\ct\,\sin^2\theta\Big)-i\frac{(-2\eps i)\sin^2\theta}{\st}e^{i\phi}\\
        &=\eps \,e^{i\phi}\big(32\cos^2\theta\,\st-10\st\big)\\
        &=\epsilon\Big[\frac{32}{5}Z_3-\frac{18}{5}Z_1\Big]. 
    \end{split}
\end{equation}

We next turn to the term $\xi\nabla \bV$:
\begin{equation}\label{GFL101}
\xi\nabla(Ve_\theta)=\xi^\theta\partial_\theta V e_\theta+\xi^\phi V \cot\theta\,e_\phi, \quad \xi=\xi^\theta e_\theta+\xi^\phi e_\phi=-\frac{\partial Z_2}{\st\,\partial\phi} e_\theta+\frac{\partial Z_2}{\partial\theta} e_\phi. 
\end{equation}
We have
\begin{equation}\label{GFL102}
    \begin{split}
        {\rm div}\,(\xi\nabla Ve_\theta)&=(-2\eps){\rm div}\,\big(\ct\,\nabla^\perp Z_2\big)=(-2\eps)\nabla (\ct)\cdot \nabla^\perp Z_2\\
        &=(-2\epsilon)(-\st)\Big(-\frac{\partial Z_2}{\st\,\partial\phi}\Big)=(-2\eps i)Z_2,
    \end{split}
\end{equation}
and 
\begin{equation}\label{GFL103}
    \begin{split}
        {\rm curl}\,(\xi\nabla {\bf V})&=(-2\eps)\,{\rm curl}\,\big(\ct\,\nabla^\perp Z_2\big)\\
        &=(-2\eps)(\nabla \ct)\cdot (\nabla Z_2)+(-2\eps)\ct\,(\Delta Z_2)\\
        &=(-2\eps)(-\st\,e_\theta)\cdot\Big(\frac{\partial Z_2}{\partial\theta}e_\theta+\frac{\partial Z_2}{\st\partial\phi}e_\phi\Big)+(-2\eps)\ct\,(-6Z_2)\\
        &=(-2\eps \,e^{i\phi})\Big[-\st(2\cos^2\theta-1)-6\cos^2\theta\,\st\Big]\\
        &=\eps\Big[\frac{16}{5}Z_3+\frac{6}{5}Z_1\Big].
    \end{split}
\end{equation}

We now turn to the term
\begin{equation}\label{GFL104}
    F\xi'=F(-\nabla^\perp Z_2)=-F\nabla^\perp Z_2.
\end{equation}
We have 
\begin{equation}\label{GFL105}
    \begin{split}
        {\rm div}\,(F\xi')&=-{\rm div}\,\big(F\nabla^\perp Z_2\big)=-\frac{\partial F}{\partial\theta}\Big(-\frac{\partial Z_2}{\st\,\partial\phi}\Big)\\
        &=\eps \big(-4i\st\,\ct\,e^{i\phi}\big)=\eps (-4i)Z_2,
    \end{split}
\end{equation}
and
\begin{equation}\label{GFL106}
    \begin{split}
        {\rm curl}\,(F\xi')&=-{\rm curl}\,\big(F\nabla^\perp Z_2\big)=-\nabla F\cdot\nabla Z_2-F\Delta Z_2\\
        &=-\frac{\partial F}{\partial\theta}\frac{\partial Z_2}{\partial\theta}+6FZ_2=4\eps\,\st(2\cos^2\theta-1)+6FZ_2\\
        &=\epsilon (32)\cos^2\theta\,\st-\eps \,4\st=\eps \Big(\frac{32}{5}(5\cos^2\theta-1)+\frac{32}{5}-4\Big)\st\\
        &=\epsilon\,\frac{32}{5}Z_3+\epsilon\,\frac{12}{5}Z_1.
    \end{split}
\end{equation}

We summarize the calculations in Table \ref{table6}.

\begin{table}
    \centering
    \begin{tabular}{c|c|c|c|c|}
         & ${\bf V}\nabla \xi$ &$\xi\nabla\bV$  & $F\xi'$ & $-\bV\nabla\xi-\xi\nabla\bV-F\xi'$\\
         \hline
       div  & $\eps\,2i\,Z_2$  & $\epsilon (-2i)Z_2$ & $\epsilon (-4i)Z_2$  & $\epsilon (4i)Z_2$ \\
     curl & $\eps\big[\frac{32}{5}Z_3-\frac{18}{5}Z_1\big]$   &$\eps\big[\frac{16}{5}Z_3+\frac{6}{5}Z_1\big]$  &$\eps\big[\frac{32}{5}Z_3+\frac{12}{5}Z_1\big]$    &$-\eps\,16Z_3$ \\
    \end{tabular}
    \caption{Summary of Calculations}
    \label{table6}
\end{table}

Recalling the definition \eqref{S31e1} for the operator $K$, it remains to evaluate $-\xi\nabla F+2\xi\cdot(Ve_\theta)$ (with $\xi=\nabla^\perp Z_2$). We have
\begin{equation}\label{GFL107}
    -\xi\nabla F+2\xi\cdot(Ve_\theta)=-\xi^\theta \partial_\theta F+2\xi^\theta V=\xi^\theta (-\partial_\theta F+2V)=O(\epsilon^2).
\end{equation}
In the above calculation we also used the formulae \eqref{GFL56}. 

We see that the contribution to $P_1K\begin{pmatrix}
    \xi\\\xi'\\0\\0
\end{pmatrix}$ comes from the term $\eps \,4iZ_2$ in Table \ref{table6}. The function $Z_3$ in the table does not generate any eigenvector of $L_0$ with $\lambda=1$. 

The potential $\varphi$ corresponding to ${\rm div}=\epsilon(4i)Z_2$ is
\begin{equation}\label{GFL108}
    \varphi=-\frac{1}{6}\epsilon\,4iZ_2=\epsilon \big(-\frac{2}{3}i\big)Z_2.
\end{equation}
Given that the stream function parts does not play a role in this contribution, we can assume they are zero. We need to decompose the vector $\begin{pmatrix}
    0\\0\\\epsilon\big(-\frac{2}{3}i\big)Z_2\\0\\0\\0
\end{pmatrix}$ into expansion of eigenvectors of $L_0$. We use Table \ref{table3} and the identity
\begin{equation}\label{GFL109}
    \begin{pmatrix}
        0\\1\\0\\0
    \end{pmatrix}=\frac{3}{7}\begin{pmatrix}
        -\frac{1}{3}\\1\\1\\-3
    \end{pmatrix}-\frac{3}{5}\begin{pmatrix}
        0\\0\\1\\-3
    \end{pmatrix}+\frac{6}{35}\begin{pmatrix}
        \frac{5}{6}\\\frac{10}{3}\\1\\-3
    \end{pmatrix}.
\end{equation}
As a consequence, we see that
\begin{equation}\label{GFL110}
    P_1\begin{pmatrix}
    0\\0\\\epsilon\big(-\frac{2}{3}i\big)Z_2\\0\\0\\0
\end{pmatrix}=\epsilon \,\frac{2}{5}i\begin{pmatrix}
    0\\0\\0\\0\\Z_2\\-3Z_2
\end{pmatrix}.
\end{equation}
Theorefore,
\begin{equation}\label{GFL111}
    P_1K\begin{pmatrix}
        0\\\epsilon Z_2\\0\\-\epsilon Z_2\\0\\0
    \end{pmatrix}=\epsilon^2\,\frac{2}{5}i\begin{pmatrix}
    0\\0\\0\\0\\Z_2\\-3Z_2
\end{pmatrix}.
\end{equation}
We conclude that (recall \eqref{GFL41} for the definition of $\Xi_1'$)
    \begin{equation}\label{GFL112}
      B(I-\Lambda)^{-1}C\,\Xi_1=\epsilon^2i\big(-\frac{19}{5}\big)\Xi_1',  
    \end{equation}
and hence
\begin{equation}\label{GFL113}
    A\,\Xi_1+B(I-\Lambda)^{-1}C\,\Xi_1=P_1K\,\Xi_1+B(I-\Lambda)^{-1}C\,\Xi_1=\epsilon^2\big(-\frac{1}{5}i\Xi_1'\big).
\end{equation}

\subsubsection{Calculations on the second eigenvector $\Xi_1'$} Recall from \eqref{GFL61} that the vector
\begin{equation}\label{GFL114}
    \Xi=\begin{pmatrix}
0\\\epsilon\,2iZ_1+\epsilon^2\,2iZ_2\\0\\-2iZ_1-\epsilon^2\,2iZ_2\\-\epsilon \,6Z_2+\eps^2\big[-\frac{16}{5}Z_3-\frac{6}{5}Z_1\big]\\\epsilon \,18Z_2+\eps^2\big[\frac{48}{5}Z_3-\frac{2}{5}Z_1\big]
\end{pmatrix}+O(\eps^3).
\end{equation}
(calculated up to $O(\eps^3)$) is an eigenvalue of $L$ with $\lambda=1$. The projection $P_1\Xi$ is given by
\begin{equation}\label{GFL115}
    P_1\Xi=\epsilon\,2i\Xi_1-\epsilon\,6\Xi_1'+O(\eps^3),
\end{equation}
as the terms with $\eps^2$ do not contribute to the projection to $E$ (see Table \ref{tab1}). 

The equation $L\Xi=\Xi$ can be expressed in the equation in $E$,
\begin{equation}\label{GFL116}
   \big(A+B(I-\Lambda)^{-1}C\big)\big(2i\Xi_1-6\Xi_1'\big)=O(\epsilon^3). 
\end{equation}
Denote
\begin{equation}\label{GFL117}
    \mathcal{O}=A+B(I-\Lambda)^{-1}C.
\end{equation}
From \eqref{GFL116}, we have
\begin{equation}\label{GFL118}
    \mathcal{O}\,\Xi_1'=\frac{i}{3}\mathcal{O}\Xi_1+\mathcal{O}(\epsilon^3),
\end{equation}
and we previously calculated (see \eqref{GFL113})
\begin{equation}\label{GFL119}
   \mathcal{O}\,\Xi_1=\epsilon^2\Big(-\frac{i}{5}\Xi_1'\Big).
\end{equation}
Hence
\begin{equation}\label{GFL120}
    \mathcal{O}\,\Xi_1'=\eps^2\,\frac{1}{15}\Xi_1',
\end{equation}
and we see that the matrix of $\mathcal{O}$ in the basis $\Xi_1, \Xi_1'$ is 
\begin{equation}\label{GFL121}
    \eps^2\begin{pmatrix}
        0&0\\
        -\frac{i}{5}&\frac{1}{15}
    \end{pmatrix}+\mathcal{O}(\epsilon^3)
\end{equation}
which has eigenvalue $0$ and $\frac{1}{15}\eps^2$ (up to order $O(\eps^3)$). 

Therefore, recalling \eqref{GFL11} and \eqref{GFL117}, we conclude that $L\sim I+A+B(I-\Lambda)^{-1}C$ has eigenvalues $\lambda=1+O(\epsilon^3)$ and $\lambda=1+\frac{1}{15}\eps^2+O(\epsilon^3)$ when restricted to $X^m$ with $m=1$, and by \eqref{GFL121}, the eigenvectors form an angle $\gtrsim1$. We remark that the eigenvalue $\lambda=1+O(\epsilon^3)$ is actually $\lambda=1$ since $\Xi$ from \eqref{GFL114} is an eigenvector.

\subsection{The mode $m=2$}
Here the situation is simpler, as the unperturbed eigenspace is one-dimensional. Let us denote by $\Xi_2$ the corresponding eigenvector. It is of the form
\begin{equation}\label{GFL122}
    \Xi_2=\begin{pmatrix}
        0\\0\\Z_2^2\\-3Z_2^2
    \end{pmatrix},\quad Z_2^2=\sin^2\theta\,e^{2i\phi}. 
\end{equation}
A calculation similar to those above gives 
\begin{equation}\label{GFL123}
    K\,\Xi_2=\epsilon^2\frac{4}{15}\Xi_2+O(\epsilon^3). 
\end{equation}

\bibliographystyle{plain} 
\bibliography{bibliography}

@article {SverakKorolev,
    AUTHOR = {Korolev, A. and \v{S}ver\'{a}k, V.},
     TITLE = {On the large-distance asymptotics of steady state solutions of
              the {N}avier-{S}tokes equations in 3{D} exterior domains},
   JOURNAL = {Ann. Inst. H. Poincar\'{e} C Anal. Non Lin\'{e}aire},
  FJOURNAL = {Annales de l'Institut Henri Poincar\'{e} C. Analyse Non
              Lin\'{e}aire},
    VOLUME = {28},
      YEAR = {2011},
    NUMBER = {2},
     PAGES = {303--313},
      ISSN = {0294-1449,1873-1430},
   MRCLASS = {35Q30 (35B40 76D05)},
  MRNUMBER = {2784073},
MRREVIEWER = {Lorenzo\ Brandolese},
       DOI = {10.1016/j.anihpc.2011.01.003},
       URL = {https://doi.org/10.1016/j.anihpc.2011.01.003},
}

@article {Batchelor,
    AUTHOR = {Batchelor, G.K.},
     TITLE = {An introduction to fluid dynamics},
   JOURNAL = {Cambridge University Press},
  
      YEAR = {1974},
    }

@article {Lifschitz,
    AUTHOR = {Landau, L.D. and Lifshitz, E.M.},
     TITLE = {Fluid mechanics},
   JOURNAL = {Butterworth-Heinemann; 2nd edition 
},
  
      YEAR = {1987},
    }

@article{amick1988leray,
  author  = {Amick, C. J.},
  title   = {On {L}eray's problem of steady {N}avier--{S}tokes flow past a body in the plane},
  journal = {Acta Math.},
  volume  = {161},
  pages   = {71--130},
  year    = {1988}
}

@article{babenko1973stationary,
  author  = {Babenko, K. I.},
  title   = {On stationary solutions of the problem of flow past a body of a viscous incompressible fluid},
  journal = {Mat. Sb.},
  volume  = {91},
  number  = {133},
  pages   = {3--25},
  year    = {1973},
  note    = {English Translation: \textit{Math. USSR Sbornik}, 20, 1973, 1--25}
}

@article{cannone2004smooth,
  author  = {Cannone, M. and Karch, G.},
  title   = {Smooth or singular solutions to the {N}avier--{S}tokes system?},
  journal = {J. Differential Equations},
  volume  = {197},
  number  = {2},
  pages   = {247--274},
  year    = {2004}
}

@article{deuring2000asymptotic,
  author  = {Deuring, P. and Galdi, G. P.},
  title   = {On the asymptotic behavior of physically reasonable solutions to the stationary {N}avier--{S}tokes system in three-dimensional exterior domains with zero velocity at infinity},
  journal = {J. Math. Fluid Mech.},
  volume  = {2},
  number  = {4},
  pages   = {353--364},
  year    = {2000}
}

@article{finn1965exterior,
  author  = {Finn, R.},
  title   = {On the exterior stationary problem for the {N}avier--{S}tokes equations, and associated perturbation problems},
  journal = {Arch. Rational Mech. Anal.},
  volume  = {19},
  pages   = {363--406},
  year    = {1965}
}

@book{galdi1994introduction,
  author    = {Galdi, G. P.},
  title     = {An Introduction to the Mathematical Theory of the {N}avier--{S}tokes Equations: Volumes I and II},
  publisher = {Springer},
  year      = {1994}
}

@article{landau1944new,
  author  = {Landau, L. D.},
  title   = {A new exact solution of the {N}avier--{S}tokes equations},
  journal = {Dokl. Akad. Nauk SSSR},
  volume  = {43},
  pages   = {299},
  year    = {1944}
}

@article{leray1933etude,
  author  = {Leray, J.},
  title   = {\'{E}tude de diverses \'{e}quations int\'{e}grales non lin\'{e}aires et de quelques probl\`{e}mes que pose l'hydrodynamique},
  journal = {J. Math. Pures Appl.},
  volume  = {12},
  pages   = {1--82},
  year    = {1933}
}

@article{nazarov2000steady,
  author  = {Nazarov, S. A. and Pileckas, K.},
  title   = {On steady {S}tokes and {N}avier--{S}tokes problems with zero velocity at infinity in a three-dimensional exterior domain},
  journal = {J. Math. Kyoto Univ.},
  volume  = {40},
  number  = {3},
  pages   = {475--492},
  year    = {2000}
}

@article{tian1998one,
  author  = {Tian, G. and Xin, Z.},
  title   = {One-point singular solutions to the {N}avier--{S}tokes equations},
  journal = {Topol. Methods Nonlinear Anal.},
  volume  = {11},
  number  = {1},
  pages   = {135--145},
  year    = {1998}
}

@article{Yud67,
  author  = {Yudovich, V. I.},
  title   = {An example of loss of stability and generation of a secondary flow in a closed vessel},
  journal = {Math. USSR-Sb.},
  volume  = {3},
  number  = {4},
  pages   = {519--533},
  year    = {1967},
  doi     = {10.1070/SM1967v003n04ABEH002764}
}

@book{tsai2018lectures,
  author    = {Tsai, T.P.},
  title     = {Lectures on {N}avier--{S}tokes Equations},
  series    = {Graduate Studies in Mathematics},
  volume    = {192},
  publisher = {American Mathematical Society},
  year      = {2018},
  isbn      = {978-1470430962}
}

@article{jia2018asymptotics,
  author  = {Jia, H. and \v{S}ver\'{a}k, V.},
  title   = {Asymptotics of stationary {N}avier--{S}tokes equations in higher dimensions},
  journal = {Acta Mathematica Sinica, English Series},
  volume  = {34},
  number  = {4},
  pages   = {598--611},
  year    = {2018}
}

@article{sverak2011landau,
  author  = {\v{S}ver\'{a}k, V.},
  title   = {On {L}andau's solutions of the {N}avier--{S}tokes equations},
  journal = {Journal of Mathematical Sciences},
  volume  = {179},
  pages   = {208--228},
  year    = {2011},
  doi     = {10.1007/s10958-011-0590-5}
}

@article{guillod2015steady,
  author  = {Guillod, J.},
  title   = {Steady solutions of the {N}avier--{S}tokes equations in the plane},
  journal = {arXiv preprint arXiv:1511.03938},
  year    = {2015},
  eprint  = {1511.03938},
  archivePrefix = {arXiv},
  primaryClass  = {math.AP},
  doi     = {10.48550/arXiv.1511.03938}
}

@article{korobkov2023stationary_ren,
  author  = {Korobkov, M. and Ren, X.},
  title   = {Stationary Solutions to the {N}avier--{S}tokes System in an Exterior Plane Domain: 90 Years of Search, Mysteries and Insights},
  journal = {J. Math. Fluid Mech.},
  volume  = {25},
  pages   = {55},
  year    = {2023},
  doi     = {10.1007/s00021-023-00792-w}
}

@article{korobkov2015solution,
  author  = {Korobkov, M. and Pileckas, K. and Russo, R.},
  title   = {Solution of {L}eray's problem for stationary {N}avier--{S}tokes equations in plane and axially symmetric spatial domains},
  journal = {Annals of Mathematics},
  volume  = {181},
  number  = {2},
  pages   = {769--807},
  year    = {2015}
}

@book{constantin1988navier,
  author    = {Constantin, P. and Foia\c{s}, C.},
  title     = {{N}avier--{S}tokes Equations},
  publisher = {University of Chicago Press},
  year      = {1988},
  pages     = {190}
}

@article{farwig2009asymptotic,
  author  = {Farwig, R. and Hishida, T.},
  title   = {Asymptotic profiles of steady {S}tokes and {N}avier--{S}tokes flows around a rotating obstacle},
  journal = {Ann. Univ. Ferrara},
  volume  = {55},
  pages   = {263--277},
  year    = {2009},
  doi     = {10.1007/s11565-009-0072-6}
}

\end{document}